\documentclass[11pt]{amsart}

\usepackage{amsmath,amssymb, amsthm,tikz-cd,pdflscape,color,stmaryrd,hyperref}

\usepackage[utf8]{inputenc}
\usepackage{thmtools}
\usetikzlibrary{bending}
\usepackage{adjustbox}
\usepackage[margin=1in]{geometry}

\usepackage[english]{babel}

\hypersetup{
    colorlinks,
    citecolor=blue,
    filecolor=blue,
    linkcolor=blue,
    urlcolor=blue
}

\usepackage[toc,page]{appendix}

\usepackage{rotating}
\usepackage{enumerate}
\newcommand{\C}{{\mathbb C}}
\newcommand{\stable}{{\mathrm{st}}}
\newcommand{\unst}{{\mathrm{unst}}}

\newcommand{\F}{{\mathbb F}}

\newcommand{\mcH}{{\mathcal H}}
\newcommand{\G}{{\mathbb G}}
\newcommand{\cG}{{\mathcal G}}
\newcommand{\cL}{{\mathcal L}}

\newcommand{\cO}{{\mathfrak o}}

\newcommand{\mcQ}{{\mathcal Q}}

\newcommand{\Oo}{\mathcal{O}}

\newcommand{\p}{{\mathfrak p}}

\newcommand{\jacquet}{{\mathrm{r}}}

\newcommand{\SL}{{\mathrm{SL}}}

\newcommand{\GL}{{\mathrm{GL}}}
\newcommand{\SO}{{\mathrm{SO}}}

\newcommand{\tO}{{\mathrm{O}}}

\newcommand{\supercusp}{{\mathrm{s.c.}}}

\newcommand{\PGL}{{\mathrm{PGL}}}

\newcommand{\Cent}{{\mathrm{Z}}}

\DeclareMathOperator{\diag}{diag}
\DeclareMathOperator{\Ind}{Ind}
\DeclareMathOperator{\cInd}{c-Ind}

\DeclareMathOperator{\tr}{tr}

\DeclareMathOperator{\sgn}{sgn}

\DeclareMathOperator{\St}{St}

\newcommand{\cusp}{\mathrm{cusp}}
\newcommand{\princ}{\mathrm{princ}}

\newcommand{\Ch}{{\mathrm{Ch}}}

\newcommand{\rss}{{\mathrm{rss}}}

\numberwithin{equation}{subsection}

\newtheorem{thm}{Theorem}[subsection]
\newtheorem*{thm*}{Theorem}
\newtheorem{theorem}[thm]{Theorem}
\newtheorem{corollary}[thm]{Corollary}
\newtheorem{lemma}[thm]{Lemma}
\newtheorem{proposition}[thm]{Proposition}

\theoremstyle{definition}
\newtheorem{definition}[thm]{Definition}

\newtheorem{remark}[thm]{Remark}

\newtheorem{property}[equation]{Property}
\theoremstyle{definition}

\title[Local Langlands Correspondence for $G_2$]{The explicit Local Langlands Correspondence for $G_2$ II: character formulas and stability}

\author{Kenta Suzuki}
\address{M.I.T., 77 Massachusetts Avenue,
Cambridge, MA, USA}
\email{kjsuzuki@mit.edu}

\author{Yujie Xu}
\address{M.I.T., 77 Massachusetts Avenue,
Cambridge, MA, USA}
\email{yujiexu@mit.edu}

\begin{document}

\maketitle

\begin{abstract}
We write down character formulas for 
representations of $G_2$ considered in \cite{AX-LLC}, and show that stability for $L$-packets uniquely pins down the Local Langlands Correspondence constructed in \cite{AX-LLC}, thus proving unique characterization of the LLC \textit{loc.cit.}  
\end{abstract}

\tableofcontents

\section{Introduction}
In this article, we complete the unique characterization of the explicit local Langlands correspondence for $p$-adic $G_2$ constructed in \cite{AX-LLC}. More precisely, we use stability property of $L$-packets to uniquely pin down the choices of twists in the $L$-packets from \cite{AX-LLC}.

The rough idea is as follows: we explicitly calculate Harish-Chandra characters for the representations (including non-supercuspidals) in certain neighborhoods of semisimples in $G_2$ (see for example \S \ref{section: Green-fxn-pi-eta2}, \S\ref{character-2x2-semisimple-section}, \S\ref{3x3-character-section} and \S\ref{3x3-character-section-ss}). 
In particular, stability property \ref{property:atomic-stability} (as formulated by DeBacker and Kaletha) 
implies the stability of the sum of characters in an $L$-packet locally around each semisimple. Using \cite{DeBacker-Kazhdan-G2} (which builds on some works of Waldspurger), we deduce that the sum of two specific characters (one for a non-supercuspidal and another one for a \textit{singular} supercuspidal) are stable, thus pinning down the size $2$ mixed packets in \cite{AX-LLC} (see Theorem \ref{stability-theorem-2x2}). The size $3$ mixed packets are pinned down similarly (see Theorem \ref{stability-theorem-3x3} and Theorem \ref{stability-theorem}). Our computations involve a refinement of Roche's Hecke algebra isomorphisms (see~\S\ref{roche-iso-section}).

\section{Preliminaries}\addtocontents{toc}{\protect\setcounter{tocdepth}{1}}

Let $\pi$ be an admissible representation of $G_2$, which gives rise to a distribution $\Ch_\pi$ on $C_c^\infty(G_2)$. Then \cite[Theorem~16.3]{harish-chandra-local-character} shows that $\Ch_\pi$ can be represented by a locally constant function on $G_2^{\rss}$, the regular semisimple locus in $G_2$.

\subsection{Stability of $L$-packets}
\begin{property}[DeBacker, Kaletha]\label{property:atomic-stability}
Let $\varphi$ be a discrete $L$-parameter. There exists a non-zero $\C$-linear combination 
\begin{equation}\label{stable-distribution-Lpacket}
\sum\limits_{\pi\in\Pi_{\varphi}}\dim(\rho_\pi)\Ch_{\pi},\quad\text{for }z_{\pi}\in\C,
\end{equation}
which is stable. In fact, one can take $z_\pi=\dim(\rho_\pi)$ where $\rho_\pi$ is the enhancement of the $L$-parameter. Moreover, no proper subset of $\Pi_{\varphi}$ has this property. 
\end{property}

\subsection{Parahoric subgroups}\label{parahoric-subsection}

We fix the choice of the following parahoric subgroups in $G_2(F)$, as in Diagram~\ref{fig:beta-parahoric} where the blue nodes are the roots multiplied by $\p$ in the unipotent radical $G_{x+}$.
 
 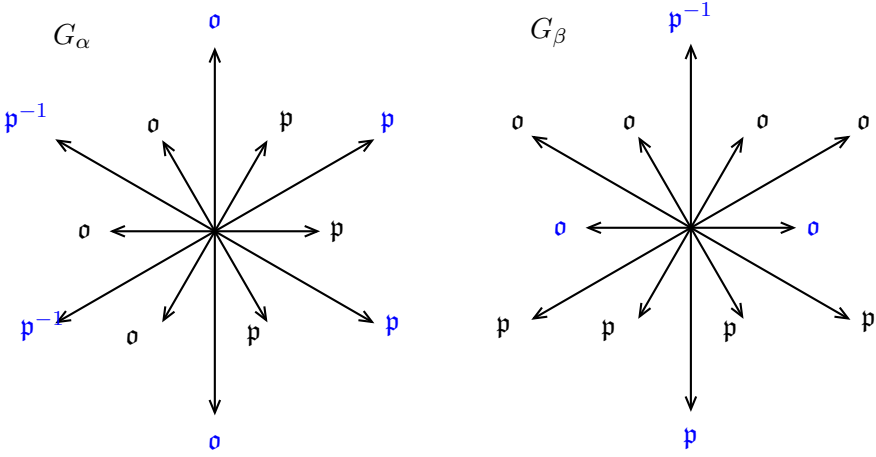
\begin{figure}
\begin{center} 
\begin{tikzpicture}
[ -{Straight Barb[bend,
       width=\the\dimexpr10\pgflinewidth\relax,
       length=\the\dimexpr12\pgflinewidth\relax]},]
    \foreach \i in {0, 1, ..., 5} {
      \draw[thick, black] (0, 0) -- (\i*60:2*0.7);
      \draw[thick, black] (0, 0) -- (30 + \i*60:3.5*0.7);
    }
    \node[right] at (2*0.7, 0) {$\p$};
    \node[above left] at (45:2.4*0.7) {$\p$};
    \node[above right] at (-48.5:-2.25*0.7) {$\cO$};
    \node[above right] at (51:-3*0.7) {$\cO$};
    \node[above right] at (-80:2.4*0.7) {$\p$};
    \node[right] at (-2.8*0.7, 0) {$\cO$};
    \node[blue] at (0,4*0.7) {$\cO$};
    \node[blue] at (0, -4*0.7) {$\cO$};
    \node[blue,above left, inner sep=.2em] at (5*30:3.5*0.7) {$\p^{-1}$};
    \node[blue,above left, inner sep=.2em] at (5*30:-4.2*0.7) {$\p$};
    \node[blue,above right, inner sep=.2em] at (-5*30:4.4*0.7) {$\p^{-1}$};
    \node[blue,above right, inner sep=.2em] at (-5*30:-3.5*0.7) {$\p$};
    \node[left] at (-60:-3) {$G_\alpha$};
     \end{tikzpicture}\hspace{1cm}
\begin{tikzpicture}
   [ -{Straight Barb[bend,
       width=\the\dimexpr10\pgflinewidth\relax,
       length=\the\dimexpr12\pgflinewidth\relax]},]
    \foreach \i in {0, 1, ..., 5} {
      \draw[thick, black] (0, 0) -- (\i*60:2*0.7);
      \draw[thick, black] (0, 0) -- (30 + \i*60:3.5*0.7);
    }
    \node[blue,right] at (2*0.7, 0) {$\cO$};
    \node[above left] at (45:2.4*0.7) {$\cO$};
    \node[above right] at (-48.5:-2.25*0.7) {$\cO$};
    \node[above right] at (51:-3*0.7) {$\p$};
    \node[above right] at (-80:2.4*0.7) {$\p$};
    \node[blue,right] at (-2.8*0.7, 0) {$\cO$};
    \node[blue] at (0,4*0.7) {$\p^{-1}$};
    \node[blue] at (0, -4*0.7) {$\p$};
    \node[above left, inner sep=.2em] at (5*30:3.5*0.7) {$\cO$};
    \node[above left, inner sep=.2em] at (5*30:-4.2*0.7) {$\p$};
    \node[above right, inner sep=.2em] at (-5*30:4.4*0.7) {$\p$};
    \node[above right, inner sep=.2em] at (-5*30:-3.5*0.7) {$\cO$};
    \node[left] at (-60:-3) {$G_\beta$};
     \end{tikzpicture}
\caption{The parahoric subgroups $G_\alpha$ and $G_\beta$}
\label{fig:beta-parahoric}
\end{center}
\end{figure}
 
Non-canonically (i.e., given a choice of uniformizer) there are isomorphisms $G_\alpha/G_{\alpha+}\cong\SL_3(\F_q)$ and $G_\beta/G_{\beta+}\cong\SO_4(\F_q)$,

More canonically, we can identify $G_\alpha/G_{\alpha+}$ the reductive quotient of the parahoric of $\SL_3$:
\begin{equation}
    H_\alpha:=\bigg\{g\in\begin{pmatrix}\cO&\cO&\p^{-1}\\\cO&\cO&\p^{-1}\\\p&\p&\cO\end{pmatrix}:\det g=1\bigg\}.
\end{equation}
Similarly,
\begin{equation}\label{h-beta-defn}
H_\beta:=\Big\{(g,h)\in\begin{pmatrix}\cO&\cO\\\cO&\cO\end{pmatrix}\times\begin{pmatrix}\cO&\p^{-1}\\\p&\cO\end{pmatrix}:\det(g)=\det(h)\Big\}/\cO_F^\times
\end{equation}
is a parahoric subgroup of $\SO_4(F)$, and there is a canonical isomorphism $H_\beta/H_{\beta+}\cong G_\beta/G_{\beta+}$ induced by the inclusion $\SO_4(F)\subset G_2(F)$.

\subsection{Refining Roche's isomorphism}\label{roche-iso-section}

Let $G$ be a connected split reductive group over $F$ with maximal torus $T$, and let $T_0\subset T$ be the maximal compact subgroup. Given a character $\chi\colon T_0\to\C^\times$, let $\chi^\vee\colon\cO_F^\times\to T^\vee(\C)$ be the dual, and let $H$ be a split reductive group over $F$ with maximal torus $T$ such that $H^\vee=\Cent_{G^\vee}(\mathrm{im}(\chi^\vee))$, where we assume $\Cent_{G^\vee}(\mathrm{im}(\chi^\vee))$ is connected.

Roche \cite[Thm~8.2]{Roche-principal-series} produces a support-preserving isomorphism $\mcH(G/\!/I,\chi)\cong \mcH(H/\!/J,1)$ where $I$ is an Iwahori subgroup of $G$ and $J$ is an Iwahori subgroup of $H$, but it is non-canonical. We make the isomorphism more canonical by slightly modifying the right-hand side:
\begin{proposition}\label{Roche-reformulation}
There is a unique support preserving isomorphism $\mcH(G/\!/I,\chi)\cong\mcH(H/\!/J,\chi)$ such that the following diagram commutes:
\[ \begin{tikzcd}
\mcH(T/\!/T_0,\chi) \arrow[equals]{r}{} \arrow[swap,hook]{d}{t_u} & \mcH(T/\!/T_0,\chi) \arrow[hook]{d}{t_u} \\
\mcH(G/\!/I,\chi)\arrow{r}{\sim}&\mcH(H/\!/J,\chi),
\end{tikzcd}
\]
where $t_u=t_{\delta_B^{-1/2}}$ is as in \cite[pg~399]{Roche-principal-series}.
\end{proposition}
\begin{proof}
Let $\overline H^\vee:=H^\vee/\Cent(H^\vee)$, so we have a cover $\overline H\xrightarrow{\pi} H$. Let $\overline T^\vee:=T^\vee/\mathrm{im}(\chi^\vee)$ be a maximal torus of $\overline H^\vee$, which gives rise to a maximal torus $\overline T\subset \overline H$. For some finite discrete group $g$ we have the exact sequence of algebraic groups
\[
1\to \Cent_{\overline H}\to \overline T\xrightarrow{\pi} T\to 1
\]
where since $\mathrm{im}(\chi^\vee)\subset\Cent_{H^\vee}$ the composition $\pi^\vee\circ \chi^\vee\colon \cO_F^\times\to \overline T^\vee$ is trivial, we also have that $\chi\circ\pi=1$. Thus, $\chi$ factors through $H^1_{\mathrm{gal}}(F,\Cent_{\overline H})$, and so can be viewed as a character of $H$, since $H/\pi(\overline H)\cong H^1_{\mathrm{gal}}(F,\Cent_{\overline H})$.

By \cite[Thm~6.3]{Roche-principal-series} there is a unique support-preserving homomorphism $\mcH(\overline H/\!/\overline J,1)\hookrightarrow\mcH(G/\!/I,\chi)$, which extends\footnote{a priori the extension is non-canonical, but there is a unique choice making the diagram commute} to a support-preserving isomorphism $i\colon \mcH(H/\!/J,\chi)\xrightarrow\sim\mcH(G/\!/I,\chi)$. The restriction of $i$ to $\mcH(T/\!/T_0,\chi)$ is then trivial on $\mcH(\overline T/\!/\overline T_0,1)$, so it is given by twisting by a character of $T/\pi(\overline T)$. Since $T/\pi(\overline T)\cong H/\pi(\overline H)$ such twists extend to the entire Hecke algebra $\mcH(H/\!/J,\chi)$. Thus we have constructed an isomorphism $\mcH(G/\!/I,\chi)\cong\mcH(H/\!/J,\chi)$ satisfying the properties given. 

Uniqueness is a general observation on automorphisms of Iwahori Hecke algebras $\mcH(H/\!/J,1)$ being determined by its restriction to $\C[T/T_0]=\mcH(T/\!/T_0,1)$.
\end{proof}

\section{Size $2$ mixed packets}

Recall the size $2$ depth-zero mixed packets from \cite{AX-LLC}, where $\pi(\eta_2)$ is the principal series representation in Table 17 \textit{loc.cit.}. It is the unique (tempered) sub-representation of the parabolic induction $I_B^{G_2}(\eta_2\otimes\nu\eta_2)$, where $\eta_2$ is a ramified quadratic character of $F^\times$.

\subsection{Preliminaries on $\SO_4(F)$}\label{section-so4-prelim}

We let $\SO_4(F):=\{(g,h)\in\GL_2(F)\times\GL_2(F):\det(g)=\det(h)\}/F^\times$, where $F^\times$ is diagonally embedded as $\{(aI_2,aI_2):a\in F^\times\}$. It has a standard rank $2$ maximal torus $T:=\{(\diag(a_1,a_2),\diag(b_1,b_2)):a_1a_2=b_1b_2\}/F^\times$. Given characters $\chi_1,\chi_2,\varphi_1,\varphi_2$ of $F^\times$ such that $\chi_1\chi_2=\varphi_1\varphi_2$, we let $\chi_1\otimes\chi_2\otimes\varphi_1\otimes\varphi_2$ denote the character
\[
\chi_1\otimes\chi_2\otimes\varphi_1\otimes\varphi_2(\diag(a_1,a_2),\diag(b_1,b_2))=\chi_1(a_1)\chi_2(a_2)\varphi_1(b_1)\varphi_2(b_2).
\]
Note that for any character $\theta$ of $F^\times$, we have $\chi_1\otimes\chi_2\otimes\varphi_1\otimes\varphi_2=\theta\chi_1\otimes\theta\chi_2\otimes\theta\varphi_1\otimes\theta\varphi_2$.

By abuse of notation, let $\widetilde{\det}\colon\SO_4(F)\to F^\times/(F^\times)^2$ be defined by $\widetilde{\det}(g,h):=\det(g)=\det(h)$. Thus, for any order $2$ character $\eta$ of $F^\times$, we obtain a character $\eta\circ\widetilde{\det}$ of $\SO_4(F)$. 
The same conventions apply for $\SO_4(\cO_F)$ and $\SO_4(\F_q)$.

The generalized Springer correspondence for $\SO_4$ is given in Table~\ref{so4-springer} (see \cite[\S10.1, p.~166]{collingwood}), where $e$ denotes the regular unipotent of $\SL_2$, and $\cL$ denotes the unique nontrivial cuspidal local system on the orbit of $ee$. Let $\cG_{\sgn}$ denote the generalized Green function associated to the cuspidal local system $(ee,\cL)$, as in \cite[\S5.2.2]{DeBacker-Kazhdan-G2}.
{\begin{center}\begin{table}\begin{tabular}{ |c|c| } 
 \hline
 Unipotent pairs& Representations of $W\cong\mu_2^2$\\ \hline
 $(00,\C)$&$(1,1),1$\\
 $(0e,\C)$&$1\otimes\sgn$\\
 $(e0,\C)$&$\sgn\otimes1$\\
 $(ee,\C)$&$\sgn\otimes\sgn$\\
 $(ee,\cL)$&cuspidal\\
 \hline
\end{tabular}\caption{Springer Correspondence for $\SO_4(\C)$}\label{so4-springer}\end{table}\end{center}}

\subsection{Calculating parahoric invariants for $\pi(\eta_2)$}

\subsubsection{Calculating $\pi(\eta_2)^{G_{\beta+}}$}\label{eta2-parahoric-section}

By \cite[\S4.3]{bonnafe-book}, there are two reducible Deligne-Lusztig inductions of $\SL_2(\F_q)$: the principal series representations $R_\pm(\alpha_0)$ and the cuspidal representations $R_\pm'(\theta_0)$, where $\alpha_0$ and $\theta_0$ are the unique order $2$ character of $\F_q^\times$ and $\mu_{q+1}$, respectively (in \cite[\S2]{lusztig-book}, $R_{\pm}'(\theta_0)$ is denoted $H_\epsilon'$ and $H_\epsilon''$).

\begin{remark}\label{sl2-character}
\cite[Table~5.4]{bonnafe-book} gives the following, for $x\ne0\in\F_q$:
\begin{align}
    \tr(\begin{pmatrix}1&x\\&1\end{pmatrix},R_{\pm}(\alpha_0))&=\frac12(1\pm\epsilon(x)\sqrt{q^*})\\
    \tr(\begin{pmatrix}1&x\\&1\end{pmatrix},R_{\pm}'(\theta_0))&=\frac12(-1\pm\epsilon(x)\sqrt{q^*}),
\end{align}
where $q^*:=(-1)^{\frac{q-1}{2}}q\equiv1\pmod4$.
\end{remark}

\begin{definition}\label{omega2-defn}
Let $H_\beta$ be the parahoric defined in \eqref{h-beta-defn}, which contains the index $2$ subgroup
\begin{equation}
H_\beta^0:=\Big\{(g,h)\in\begin{pmatrix}\cO&\cO\\\cO&\cO\end{pmatrix}\times\begin{pmatrix}\cO&\p^{-1}\\\p&\cO\end{pmatrix}:\det(g)=\det(h)=1\Big\}/\pm1.
\end{equation}
For a ramified quadratic character $\eta_2$ of $F^\times$, let $\varpi\in F$ be a uniformizer such that $\eta_2(\varpi)=1$. We define the following irreducible representations of $G_\beta/G_{\beta+}\cong H_\beta/H_{\beta+}$:
\begin{align}
    \omega^{\eta_2}_{\mathrm{princ}}&:=\Ind_{G_\beta^0}^{G_\beta}(R_+(\alpha_0)\boxtimes R_+(\alpha_0)^{\mathrm{diag}(\varpi,1)})\label{pi-princ+-defn} 
    \\
    \omega^{\eta_2}_{\mathrm{cusp}}&:=\Ind_{G_\beta^0}^{G_\beta}(R_+'(\theta_0)\boxtimes R_+'(\theta_0)^{\mathrm{diag}(\varpi,1)})
\end{align}
This is independent of the choice of the uniformizer $\varpi$.
\end{definition}

\begin{remark}\label{omega-eta2-characterization}
The representation $\omega_\princ^{\eta_2}$ is an irreducible constituent of the length two representation $R_T^{\SO_4}(\epsilon\circ\widetilde\det)$, for $T\subset \SO_4$ a split torus. Similarly $\omega_\cusp^{\eta_2}$ is an irreducible constituent of the length two representation $R_{T'}^{\SO_4}(\epsilon\circ\widetilde\det)$, where $T'\subset\SO_4$ is a maximal anistropic torus. 
There are multiple ways to characterize the representations $\omega_\princ^{\eta_2}$ and $\omega_\cusp^{\eta_2}$ in the Deligne-Lusztig inductions:
\begin{enumerate}
    \item By Remark~\ref{sl2-character}, for a regular unipotent $u=(\begin{pmatrix}1&x\\&1\end{pmatrix},\begin{pmatrix}1&y\\&1\end{pmatrix})\in H_\beta$ with $x\in\cO\backslash\p$ and $y\in\p^{-1}\backslash\cO$, we have
\begin{equation}
    \tr(u,\omega_\princ^{\eta_2})=\tr(u,\omega_\cusp^{\eta_2})=\frac12(1+\eta_2(xy)q^*).
\end{equation}
\item By \cite[pg~55]{bonnafe-book}, they are characterized as irreducible components of the Gelfand-Graev representation $\Gamma_{\beta,\Oo}$ (notation as in \cite[Thm~4.5]{Barbasch-Moy-LCE}) associated to the nilpotent orbit $\Oo=\Oo_1^+$ (notation as in \cite[\S7.1]{DeBacker-Kazhdan-G2}).
\end{enumerate}
\end{remark}

We use the following Hecke algebra isomorphism from \cite{aubert-xu-Hecke-algebra,AX-LLC,Roche-principal-series}: consider two copies of $\SO_4(F)$ which are Weyl group conjugates to each other. Let $\SO_4^{(1)}$ have roots $\pm\alpha,\pm(3\alpha+2\beta)$, and let $\SO_4^{(2)}$ have roots $\pm(\alpha+\beta),\pm(3\alpha+\beta)$. The following is a corollary of Proposition \ref{Roche-reformulation}. 

\begin{corollary}\label{quadratic-hecke-iso}
Let $I$ be the standard Iwahori of $G_2$. 
There exist canonical support-preserving isomorphisms of Hecke algebras
\begin{align}\label{quadratic-hecke-iso1}
\mcH(G_2/\!/I,\epsilon\otimes\epsilon)&\cong\mcH(\SO_4^{(1)}/\!/J^{(1)},\epsilon\circ\widetilde{\det})\\
\mcH(G_2/\!/I,\epsilon\otimes1)&\cong\mcH(\SO_4^{(2)}/\!/J^{(2)},\epsilon\circ\widetilde{\det}),\label{quadratic-hecke-iso2}
\end{align}
under which the representation $\pi(\eta_2)$ corresponds to the representation $\eta_2\St_{\SO_4}$, where $J^{(i)}:=I\cap\SO_4^{(i)}$ is an Iwahori subgroup of $\SO_4^{(i)}(F)$. The isomorphisms are characterized by the following commutative diagrams
\begin{equation}
 \begin{tikzcd}
\mcH(T/\!/T_0,\epsilon\otimes\epsilon) \arrow[equal]{r} \arrow[hook]{d}{t_u} & \mcH(T/\!/T_0,\epsilon\circ \widetilde{\det}) \arrow[hook]{d}{t_u}\\
\mcH(G_2/\!/I,\epsilon\otimes\epsilon)\arrow{r}{\sim}&\mcH(\SO_4^{(1)}/\!/J^{(1)},\epsilon\circ \widetilde{\det}),
\end{tikzcd}
\end{equation}
\begin{equation}
\begin{tikzcd}
\mcH(T/\!/T_0,\epsilon\otimes1) \arrow[equal]{r} \arrow[hook]{d}{t_u} & \mcH(T/\!/T_0,\epsilon\circ\widetilde{\det}) \arrow[hook]{d}{t_u}\\
\mcH(G_2/\!/I,\epsilon\otimes1)\arrow{r}{\sim}&\mcH(\SO_4^{(2)}/\!/J^{(2)},\epsilon\circ\widetilde{\det}),
\end{tikzcd}
\end{equation}
where $t_u=t_{\delta_B^{-1/2}}$ is as in \cite[pg~399]{Roche-principal-series}.
\end{corollary}
\begin{proof}
For brevity we write down the proof for the first isomorphism; the proof for the second isomorphism is entirely analogous. By \cite[Thm~6.3 and Thm~8.2]{Roche-principal-series}, there is a canonical injection
\[
\mcH(\SL_2\times\SL_2(F)/\!/J,1)\hookrightarrow\mcH(G_2/\!/I,\epsilon\otimes\epsilon)
\]
which extends (a priori) non-canonically to an isomorphism $\mcH(\SO_4(F)/\!/J,1)\cong\mcH(G_2/\!/I,\epsilon\otimes\epsilon)$. There is, however, a unique extension to $\mcH(\SO_4(F)/\!/J,1)$ which makes $\pi(\eta_2)$ correspond to $\eta_2\St_{\SO_4}$ as in Proposition \ref{Roche-reformulation}. 

The commutative diagrams follow from looking at the Jacuqet modules: the representation $\pi(\eta_2)$ is identified with a homomorphism $\mcH(G_2/\!/I,\epsilon\otimes\epsilon)\to\C$, and the (normalized) Jacquet restriction $\jacquet_\emptyset\pi(\eta_2)=\nu\eta_2\otimes\eta_2+\nu\otimes\eta_2+\eta_2\otimes\nu$ by \cite[\S 9]{AX-LLC} (see also \cite[Prop~4.1]{Muic-G2}). By \cite[Thm~9.2]{Roche-principal-series}, the restriction of the homomorphism to $\mcH(T/\!/T_0,\epsilon\otimes\epsilon\otimes1\otimes1)$ corresponds to the $\epsilon\otimes\epsilon$-isotypic component $\nu\eta_2\otimes\eta_2$.

Analogously, the (un-normalized) Jacquet restriction of $\eta_2\St_{\SO_4^{(i)}}$ is $\jacquet_\emptyset(\eta_2\St_{\SO_4^{(i)}})=\nu^{-1/2}\eta_2\otimes\nu^{1/2}\eta_2\otimes\nu^{-1/2}\otimes\nu^{1/2}$. These two characters are equal as the maximal torus of $G_2$ and the maximal torus of $\SO_4^{(i)}$ are canonically identified.
\end{proof}

By the Mackey formula, we have an isomorphism of 
representations of $G_\beta/G_{\beta+}\cong\SO_4(\F_q)$,
\begin{equation}\label{mackey-eta2-beta}
I_B^{G_2}(\nu\eta_2\otimes\eta_2)^{G_{\beta+}}\cong\bigoplus_{w\in B\backslash G_2/G_\beta}\Ind_{G_\beta\cap wBw^{-1}/(G_{\beta+}\cap wBw^{-1})}^{G_\beta/G_{\beta+}}(\epsilon\otimes\epsilon)^w,
\end{equation}
where
\begin{equation}
B\backslash G_2/G_\beta\cong W(G_2)/W(\SO_4)=W/\langle s_\alpha,s_{3\alpha+\beta}\rangle=\{1,s_\beta,s_{3\alpha+\beta}\}.
\end{equation}
The intersections $G_\beta\cap wBw^{-1}$ are shown in the following diagram~\ref{fig:beta-parahoric}, where the blue nodes correspond to the reductive quotient of the parahoric. (Note that in $G_{\beta+}$, the blue nodes are multiplied by $\p$.) 
Therefore, the $G_{\beta+}$-invariants of $I_B(\nu\eta_2\otimes\eta_2)^{G_{\beta+}}$ gives
\begin{equation}
    I_B^{G_2}(\nu\eta_2\otimes\eta_2)^{G_{\beta+}}\simeq \Ind_B^{\SO_4}(\epsilon\otimes\epsilon\otimes 1\otimes 1)+\Ind_B^{\SO_4}(\epsilon\otimes1\otimes\epsilon\otimes 1)^2
\end{equation}
Analogously, computing the $G_{\beta+}$-invariants of $I_{\alpha}$ (resp.~$I_{\beta}$) from \cite[\S 9]{AX-LLC} gives us the following
\begin{align}
    I_{\alpha}(\nu^{1/2}\eta_2\St)^{G_{\beta+}}&\simeq \Ind_P^{\SO_4}(\epsilon\St)+\Ind_B^{\SO_4}(\epsilon\otimes1\otimes\epsilon\otimes 1)\\
    I_{\beta}(\nu^{1/2}\eta_2\St)^{G_{\beta+}}&\simeq \Ind_P^{\SO_4}(\epsilon\St)+ \Ind_B^{\SO_4}(\epsilon\otimes1\otimes\epsilon\otimes 1)
\end{align}
We pin down the $G_{\beta+}$-invariance of $\pi(\eta_2)$ in Corollary~\ref{eta2-beta-invariants}.

\begin{proposition}\label{pro-iwahori-invariants}
The $I_{+}$-invariants of $\pi(\eta_2)$ is
\[
\pi(\eta_2)^{I_+}\cong\epsilon\otimes\epsilon+1\otimes\epsilon+\epsilon\otimes1.
\]
\end{proposition}
\begin{proof}
A priori we know that
\[\pi(\eta_2)^{I_+}\hookrightarrow I(\nu\eta_2\otimes\eta_2)^{I_+}=\bigoplus_{w\in W}(\epsilon\otimes\epsilon)^w=(\epsilon\otimes\epsilon)^4+(1\otimes\epsilon)^4+(\epsilon\otimes1)^4.\]
By Lemma~\ref{quadratic-hecke-iso}, the multiplicity of $\epsilon\otimes\epsilon$ in $\pi(\eta_2)$, which is the same as the multiplicity of $\epsilon\otimes\epsilon\otimes1\otimes1$ in the representation $\eta_2\St_{\SO_4}$, is one. Thus the same holds for all of the Weyl group orbits of the character.
\end{proof}

\begin{corollary}\label{eta2-beta-invariants}
There is an isomorphism of $G_\beta/G_{\beta+}$-representations
\[
\pi(\eta_2)^{G_{\beta+}}\cong\epsilon\St_{G_\beta/G_{\beta+}}\oplus\omega_{\princ}^{\eta_2}
\]
\end{corollary}
\begin{proof}
Let $N=I_+/G_{\beta+}\subseteq G_\beta/G_{\beta+}$ be a maximal unipotent subgroup of $\SO_4(\F_q)$. Let $\omega'$ and $\omega''$ be the irreducible constituents of $\Ind_B^{\SO_4}(1\otimes\epsilon\otimes1\otimes\epsilon)$. 
By Proposition~\ref{pro-iwahori-invariants}, the $\SO_4(\F_q)$-representation $\pi(\eta_2)^{G_{\beta+}}$ has $N$-invariants $\epsilon\otimes\epsilon\otimes1\otimes1+\epsilon\otimes1\otimes\epsilon\otimes1+\epsilon\otimes1\otimes1\otimes\epsilon$. Thus
\begin{align}
\pi(\eta_2)^{G_{\beta+}}&= I_\alpha(\nu^{1/2}\eta_2\St)^{G_{\beta+}}\cap I_\beta(\nu^{1/2}\eta_2\St)^{G_{\beta+}}\\
&\subseteq\epsilon\St_{\SO_4}+\omega'+\omega''
\end{align}
must contain either just $\omega'$ or $\omega''$ (but not both), since
\[
(\omega')^N,(\omega'')^N\cong\epsilon\otimes1\otimes\epsilon\otimes1+\epsilon\otimes1\otimes1\otimes\epsilon.
\]
Thus either $\pi(\eta_2)=\epsilon\St_{\SO_4}+\omega'$ or $\pi(\eta_2)=\epsilon\St_{\SO_4}+\omega''$ as abstract representations of $\SO_4(\F_q)$. 

To further pin down the choice, let $\widetilde{\mathcal{J}}:=\mathcal{J}\rtimes \langle \begin{pmatrix}
    &1\\
    \varpi &
\end{pmatrix}\begin{pmatrix}
    &1\\
    \varpi &
\end{pmatrix}\rangle$ be the stabilizer of an alcove in the Bruhat-Tits building of $\SO_4(F)$. Then we have the 
following commutative diagram involving the support-preserving isomorphism of Lemma~\ref{quadratic-hecke-iso}:
\begin{equation}\label{Roche-square-diagram-SO4}
    \begin{tikzcd}
        \mathcal{H}(G_2/\!/\mathcal{I},\epsilon\otimes 1)\arrow[]{r}{\sim}&\mathcal{H}(\SO_4/\!/ \mathcal{J},\epsilon)\\
        \mathcal{H}(G_{\beta}/\!/\mathcal{I},\epsilon \otimes 1)\arrow[hook]{u}{}\arrow[]{r}{\sim}&\mathcal{H}(\widetilde{\mathcal{J}}/\!/\mathcal{J},\epsilon)\arrow[hook]{u}{}
    \end{tikzcd}
\end{equation}
Indeed, since \eqref{quadratic-hecke-iso1} is support-preserving, the image of $\mcH(G_\beta/\!/\mathcal I,\epsilon\otimes1)$ under the isomorphism consists of functions supported on $G_\beta\cap\SO_4(F)$. Certainly $\tilde{\mathcal J}\subset G_\beta\cap\SO_4(F)$, since elements of $\tilde{\mathcal J}$, which fixes an alcove of $\SO_4(F)$, must also fix the vertex $\beta$ in the building of $G_2$. Equality follows from observing that both $\mcH(G_\beta/\!/\mathcal I,\epsilon\otimes1)$ and $\mcH(\tilde{\mathcal J}/\!/\mathcal J,\epsilon)$ have dimension $2$. 
By the characterization in Lemma~\ref{quadratic-hecke-iso}, the restriction of $\eta_2\St_{\GL_2}$ to $\mcH(\tilde{\mathcal J}/\!/\mathcal J,\epsilon)$ is the representation $\eta_2\circ\det$ on $\tilde{\mathcal J}$. Via the bottom isomorphism, $\eta_2\circ\det$ corresponds to the representation $\omega_\princ^{\eta_2}$ of $G_\beta$. 

Thus, we conclude that $\omega_\princ^{\eta_2}$ is a constituent of $\pi(\eta_2)^{G_{\beta+}}$.
\end{proof}

\subsubsection{Calculating $\pi(\eta_2)^{G_{\alpha+}}$}

Analogous to \eqref{mackey-eta2-beta}, we have
\begin{align}\label{mackey-eta2-alpha}
\begin{split}
I_B^{G_2}(\nu\eta_2\otimes\eta_2)^{G_{\alpha+}}&\cong\bigoplus_{w\in W/W(\SL_3)}\Ind_{G_\alpha\cap wBw^{-1}/(G_{\alpha+}\cap wBw^{-1})}^{G_\alpha/G_{\alpha+}}(\epsilon\otimes\epsilon)^w\\
&=\Ind_B^{\SL_3}(\epsilon)^2.
\end{split}
\end{align}
Moreover, we have isomorphisms
\begin{align}
    I_\alpha(\nu^{1/2}\eta_2\St_{\GL_2})^{G_{\alpha+}}&=\Ind_P^{\SL_3}(\epsilon\St_{\GL_2})^2\\
    I_\beta(\nu^{1/2}\eta_2\St_{\GL_2})^{G_{\alpha+}}&=\Ind_B^{\SL_3}(\epsilon),
\end{align}
where $P\subset\SL_3$ is the parabolic subgroup with Levi $\GL_2$. The intersection is
\begin{equation}
    \pi(\eta_2)^{G_{\alpha+}}=\Ind_P^{\SL_3}(\epsilon\St_{\GL_2}).
\end{equation}

\subsubsection{Calculating $\pi(\eta_2)^{G_{\delta+}}$}

Again by a Mackey theory calculation, we have:
\begin{align}
I(\nu\eta_2\otimes\eta_2)^{G_{\delta+}}&\cong\Ind_{B(\F_q)}^{G_2(\F_q)}(\epsilon\otimes\epsilon)\\
I_\alpha(\nu^{1/2}\eta_2\St_{\GL_2})^{G_{\delta+}}&\cong\Ind_{P_\alpha(\F_q)}^{G_2(\F_q)}(\epsilon\St_{\GL_2})\\
I_\beta(\nu^{1/2}\eta_2\St_{\GL_2})^{G_{\delta+}}&\cong\Ind_{P_\beta(\F_q)}^{G_2(\F_q)}(\epsilon\St_{\GL_2}),
\end{align}
where $P_\alpha$ and $P_\beta$ denote parabolic subgroups of $G_2(\F_q)$. Thus, $\pi(\eta_2)^{G_{\delta+}}$ is the intersection of $\Ind_{P_\alpha(\F_q)}^{G_2(\F_q)}(\epsilon\St_{\GL_2})$ and $\Ind_{P_\beta(\F_q)}^{G_2(\F_q)}(\epsilon\St_{\GL_2})$, denoted $\omega^\epsilon_\princ$. In terms of Lusztig's equivalence \cite[Theorem~4.23]{Lusztig-characters-Princeton-book}, if $s\in G_2(\F_q)$ is of order $2$ such that $\Cent_{G_2(\F_q)}(s)=\SO_4(\F_q)$, we have
\begin{equation}\label{Lusztig-equivalence-omega-princ-epsilon}
\mathcal E(G_2(\F_q),s)\cong\mathcal E(\SO_4(\F_q),1),
\end{equation}
and $\omega^{\epsilon}_\princ$ corresponds to $\St_{\SO_4(\F_q)}$ under \eqref{Lusztig-equivalence-omega-princ-epsilon}. 
Thus we have the following: 
\begin{proposition}\label{eta2-parahoric} Let $\pi(\eta_2)$ be the unique sub-representation of $I(\eta_2\otimes\nu\eta_2)$. Then,
\begin{align}
    \pi(\eta_2)^{G_{\delta+}}&\cong\omega^{\epsilon}_{\princ}\\
    \pi(\eta_2)^{G_{\alpha+}}&\cong\Ind_P^{\SL_3}(\epsilon\St_{\GL_2})\\
    \pi(\eta_2)^{G_{\beta+}}&\cong\epsilon\St_{G_\beta/G_{\beta+}}+\omega^{\eta_2}_\princ.
\end{align}
\end{proposition}

\subsection{The supercuspidal representation $\pi_\supercusp(\eta_2)$}\

We denote the following depth-zero supercuspidal representation of $G_2(F)$ as 
\begin{equation}\label{defn-pi-sc-eta2}
\pi_\supercusp(\eta_2):=\cInd_{G_\beta}^{G_2}(\omega_\cusp^{\eta_2}).
\end{equation}
We may readily calculate the $G_{x+}$-invariants of the supercuspidal representations $\pi_\supercusp(\eta_2)$, for various vertices $x$ in the Bruhat-Tits building as follows:
\begin{lemma}\label{eta2-sc-parahoric}
Let $\pi_{\supercusp}(\eta_2)$ be as defined in \eqref{defn-pi-sc-eta2}. We have
\begin{align}
    \pi_\supercusp(\eta_2)^{G_{\alpha+}}&=0\\
    \pi_\supercusp(\eta_2)^{G_{\beta+}}&\cong\omega_\cusp^{\eta_2}\\
    \pi_\supercusp(\eta_2)^{G_{\delta+}}&=0
\end{align}
\end{lemma}
\begin{proof}
For each vertex $x$, by Mackey theory we have
\begin{align}
   \begin{split} \pi_\supercusp(\eta_2)^{G_{x+}}&\cong\bigoplus_{g\in G_\beta\backslash G_2/G_x}\Ind^{G_x}_{G_x\cap g^{-1}G_\beta g}((\omega_\cusp^{\eta_2})^g)^{G_{x+}\cap g^{-1}G_\beta g}\\
    &=\bigoplus_{g\in G_\beta\backslash G_2/G_x}\Ind^{G_x}_{G_x\cap G_{g^{-1}\beta}}((\omega_\cusp^{\eta_2})^g)^{G_{x+}\cap G_{g^{-1}\beta}}.
    \end{split}
\end{align}
Here,
\[
((\omega_\cusp^{\eta_2})^g)^{G_{x+}\cap G_{g^{-1}\beta}}\cong(\omega_\cusp^{\eta_2})^{G_\beta\cap G_{gx+}},
\]
which is $0$ unless $\beta=gx$ since otherwise $G_\beta\cap G_{gx+}$ will contain the unipotent radical of some parabolic subgroup of $G_\beta$, so $(\omega_\cusp^{\eta_2})^{G_\beta\cap G_{gx+}}=0$ since $\omega_\cusp^{\eta_2}$ is cuspidal.
\end{proof}

\subsection{Characters on a neighborhood of $1$}\label{section: Green-fxn-pi-eta2}
In this section, we express $\pi(\eta_2)^{G_{x+}}$ in terms of generalized Green functions (notations as in \cite{DeBacker-Kazhdan-G2}), for $x=\delta,\alpha,\beta$. 
To each Weyl group conjugacy class $[w]\in W(G)$, let $S_w$ be the unique torus in $G$ such that Frobenius acts as $w$ (i.e.~the image of $w$ under the bijection of \cite[Prop~3.3.3]{carter}). We denote $R_w^{\theta}:=R_{S_w}^{\theta}$. 
Firstly, note that 
\begin{equation}\label{Ch-St-GL2}
\Ch(\St_{\GL_2})=\frac{1}{2}(R_1^1-R_{(12)}^1).
\end{equation}
\begin{enumerate}
    \item When $F=F_{G_2}$ (i.e.~corresponding to the vertex $\delta$), we have that $\pi(\eta_2)^{G_{\delta+}}\cong\omega_{\princ}^{\epsilon}$ corresponds to $\St_{\SO_4(\F_q)}$ under Lusztig's equivalence \eqref{Lusztig-equivalence-omega-princ-epsilon}. By \eqref{Ch-St-GL2}, we have 
\begin{equation}\label{so4-steinberg-green}
    \Ch_{\St_{\SO(4)}}=\frac{1}{4}(R^1_{A_1\times\widetilde{A}_1}-R_{A_1}^1-R_{\widetilde{A}_1}^1+R_1^1).
\end{equation}
Since Lusztig's equivalence \eqref{Lusztig-equivalence-omega-princ-epsilon} preserves multiplicities, we have 
\begin{equation}\label{pi(eta2)-g2-green}
    \Ch_{\pi_{\princ}^{\epsilon}}=\frac{1}{4}(R^{\epsilon}_{A_1\times\widetilde{A}_1}-R_{A_1}^{\epsilon}-R_{\widetilde{A}_1}^{\epsilon}+R_1^{\epsilon}).
\end{equation}
Restricting to the unipotent locus, for $u\in G_2(\F_q)$ unipotent we have 
\[\Ch_{\pi_{\princ}^{\epsilon}}(u)=\frac{1}{4}(\mcQ^{F_{G_2}}_{A_1\times\widetilde{A}_1}-\mcQ_{A_1}^{F_{G_2}}-\mcQ_{\widetilde{A}_1}^{F_{G_2}}+\mcQ_1^{F_{G_2}}).\]
\item When $F=F_{A_2}$ (i.e.~corresponding to the vertex $\alpha$), we have that $\pi(\eta_2)^{G_{\alpha+}}\cong \Ind_P^{\SL_3}(\epsilon\St_{\GL_2})\in\mathcal{E}(\SL_3,\begin{pmatrix}
    -1&&\\ &-1&\\ &&1
\end{pmatrix})$ corresponds, under Lusztig's equivalence, to $\St_{\GL_2}\in\mathcal{E}(\GL_2,1)$. By \eqref{Ch-St-GL2}, we have 
\begin{equation}
    \Ch(\Ind_P^{\SL_3}(\epsilon\St_{\GL_2}))=\frac{1}{2}(R_1^{\epsilon}-R_{A_1}^{\epsilon}).
\end{equation}
Restricting to the unipotent locus, we have 
\[\Ch_{\Ind_P^{\SL_3}(\epsilon\St_{\GL_2})}=\frac{1}{2}(\mcQ_1^{F_{A_2}}-\mcQ_{A_1}^{F_{A_2}}).\]
\item When $F=F_{A_1\times\tilde A_1}$ (i.e.~corresponding to the vertex $\beta$), we have that $\pi(\eta_2)^{G_{F+}}=\epsilon\St_{\SO_4}+\omega_{\princ}^{\eta_2}$. On the unipotent locus of $\SO_4(\F_q)$ we have (in the notation of \S\ref{section-so4-prelim}):
\begin{equation*}
    \begin{cases}
        \Ch(\omega_\princ^{\eta_2})+\Ch(\omega_\princ^{\eta_2'})=R_1^1\\
        \Ch(\omega_\princ^{\eta_2})-\Ch(\omega_\princ^{\eta_2'})=q^*\mathcal{G}_{\sgn}
    \end{cases},
\end{equation*}
where $q^*$ is as defined in Remark \ref{sl2-character}. 
This implies that on the unipotents,
\begin{equation}
    \Ch_{\omega_{\princ}^{\eta_2}}=\frac{1}{2}(\mcQ_1^{F_{A_1\times\widetilde{A}_1}}\pm q^* \mathcal{G}_{\sgn}).
\end{equation}
Together with \eqref{so4-steinberg-green}, we obtain:
\begin{equation}
    \Ch_{\pi(\eta_2)^{G_{F+}}}=\frac{1}{2}(\mcQ_1^{F_{A_1\times\widetilde{A}_1}}\pm q^* \mathcal{G}_{\sgn})+\frac{1}{4}(\mcQ^{F_{A_1\times\widetilde A_1}}_{A_1\times\widetilde{A}_1}-\mcQ^{F_{A_1\times\widetilde A_1}}_{A_1}-\mcQ^{F_{A_1\times\widetilde A_1}}_{\widetilde{A}_1}+\mcQ_1^{F_{A_1\times\widetilde A_1}}).
\end{equation}

\item When $F=F_{A_1}$ or $F_{A_1}'$, we have $\pi(\eta_2)^{G_{F+}}=\frac32Q_1^{F_{A_1}}-\frac12Q_{A_1}^{F_{A_1}}$ on unipotents.

\item When $F=F_{\tilde A_1}$, then again $\pi(\eta_2)^{G_{F+}}=\frac32Q_1^{F_{\tilde A_1}}-\frac12Q_{\tilde A_1}^{F_{\tilde A_1}}$ on unipotents.

\item When $F=F_\emptyset$ then $\pi(\eta_2)^{G_{F+}}=\epsilon\otimes\epsilon+1\otimes\epsilon+\epsilon\otimes1$, so the character on unipotents is $3=3Q_{1}^{\{e\}}$.
\end{enumerate}
Similarly, we have
\begin{equation}
    \Ch(\omega_\cusp^{\eta_2})=\frac{1}{2}(\mcQ_{A_1\times\widetilde{A}_1}^{F_{A_1\times\widetilde{A}_1}}\pm q^*\mathcal{G}_{\sgn}).
\end{equation}
Therefore, we have the following:
\begin{proposition}
For any ramified quadratic characters $\eta_2$ and $\eta_2'$, the sum $\pi(\eta_2)+\pi_\supercusp(\eta_2')$ has a stable character on the topologically unipotent elements.
\end{proposition}
\begin{proof}
From the discussion above, in the notation of \cite[Table~4]{DeBacker-Kazhdan-G2}, we see that for some explicitly computable constants $c_i$,
\begin{align*}
\Ch_{\pi(\eta_2)}&=\frac18c_1(D^\stable_{A_1\times\tilde A_1}+D^\unst_{A_1\times\tilde A_1})\pm c_2 D^\stable_{(F_{A_1\times\tilde A_1,\cG_{\sgn}})}+c_3D^\stable_{A_1}+c_4D^\stable_{\tilde A_1}+c_5D^\stable_{\{e\}}\\
\Ch_{\pi_\supercusp(\eta_2)}&=\frac18c_1(D^\stable_{A_1\times\tilde A_1}-D^\unst_{A_1\times\tilde A_1})\pm c_2 D^\stable_{(F_{A_1\times\tilde A_1,\cG_{\sgn}})}.
\end{align*}
Thus, by \cite[Lemma~6.4.1]{DeBacker-Kazhdan-G2} the sum is always stable.
\end{proof}

\subsection{Characters on a neighborhood of $s\in G_2$}\label{character-2x2-semisimple-section}
Let $s\in G_2$ be order $2$ such that $\Cent_{G_2}(s)=\SO_4$. By the construction in \cite[\S7]{Adler-Korman-LCE}, the distributions $\Ch_{\pi(\eta_2)}$ and $\Ch_{\pi_\supercusp(\eta_2)}$ on $G_2$ induce distributions $\Theta_{\pi(\eta_2)}$ and $\Theta_{\pi_\supercusp(\eta_2)}$ on $(\SO_4)_{0+}$, the topologically unipotent elements in $\SO_4$, such that the attached locally constant functions are compatible (see \cite[Lemma~7.5]{Adler-Korman-LCE}). We hope to see when the sum $\Theta_{\pi(\eta_2)}+\Theta_{\pi_\supercusp(\eta_2')}$ is a stable distribution on $(\SO_4)_{0+}$.

We now look at the characters on an element of the form $su$ for $u$ topologically unipotent. They follow from computations in \S\ref{section: Green-fxn-pi-eta2}.
\begin{enumerate}
    \item When $F=F_{G_2}$, by \eqref{pi(eta2)-g2-green} and \cite[Thm~4.2]{Deligne-Lusztig}, we have for $u\in\SO_4(\F_q)$ unipotent:
    \begin{align}
    \begin{split}
        \Ch_{\pi^\epsilon_\princ}(su)=&\frac14\left(R_{S_{A_1\times\tilde A_1}}^\epsilon(su)-R_{S_{A_1}}^\epsilon(su)-R_{S_{\tilde A_1}}^\epsilon(su)+R_{S_1}^\epsilon(su)\right)\\
    =&\frac1{4|\SO_4(\F_q)|}\bigg(\sum_{gsg^{-1}\in S_{A_1\times\tilde A_1}}\epsilon(gsg^{-1})\mcQ_{S_{A_1\times\tilde A_1}}^{\SO_4}(u)-\sum_{gsg^{-1}\in S_{A_1}}\epsilon(gsg^{-1})\mcQ_{S_{A_1}}^{\SO_4}(u)\\
    &-\sum_{gsg^{-1}\in S_{\tilde A_1}}\epsilon(gsg^{-1})\mcQ_{S_{\tilde A_1}}^{\SO_4}(u)+\sum_{gsg^{-1}\in S_{1}}\epsilon(gsg^{-1})\mcQ_{S_{1}}^{\SO_4}(u)\bigg)\\
        =&\frac14\big(\mcQ_{A_1\times\tilde A_1}^{A_1\times \tilde A_1}(u)-\mcQ_{A_1}^{A_1\times\tilde A_1}(u)-\mcQ_{\tilde A_1}^{A_1\times\tilde A_1}(u)+\mcQ_1^{A_1\times\tilde A_1}(u)\big)\\&+\frac12(-1)^{\frac{q-1}2}\mcQ_{1}^{A_1\times\tilde A_1}(u)+\frac12(-1)^{\frac{q+1}2}\mcQ_{A_1\times\tilde A_1}^{A_1\times\tilde A_1}(u),
        \end{split}
    \end{align}
    where the last equality folows from the observation that $gsg^{-1}\in S$ must be an order $2$ element; there are $3$ such elements for the tori $S_{A_1\times\tilde A_1}$ and $S_1$, while there is a unique such element for the tori $S_{A_1}$ and $S_{\tilde A_1}$.
    \item When $F=F_{A_1\times \tilde A_1}$, since $s\in G_{F}$ is central, we simply have:
    \begin{equation}
        \Ch_{\pi(\eta_2)^{G_{F+}}}(su)=(-1)^{\frac{q-1}2}\frac{1}{2}(\mcQ_1^{F_{A_1\times\widetilde{A}_1}}\pm q^* \mathcal{G}_{\sgn})+\frac{1}{4}(\mcQ^{F_{A_1\times\widetilde A_1}}_{A_1\times\widetilde{A}_1}-\mcQ^{F_{A_1\times\widetilde A_1}}_{A_1}-\mcQ^{F_{A_1\times\widetilde A_1}}_{\widetilde{A}_1}+\mcQ_1^{F_{A_1\times\widetilde A_1}}).
    \end{equation}
    Similarly, we have
    \begin{equation}
        \Ch_{\pi_\supercusp(\eta_2)^{G_{F+}}}(su)=(-1)^{\frac{q+1}2}\frac12(\mcQ_{A_1\times\tilde A_1}^{F_{A_1\times\tilde A_1}}\pm q^*\cG_{\sgn}).
    \end{equation}
\end{enumerate}
Since we already know that the character of $\St_{\SO_4}$ is stable, we hope to see whether $\Theta_{\pi(\eta_2)}+\Theta_{\pi_\supercusp(\eta_2)}-\Ch_{\St_{\SO_4}}$ or $\Theta_{\pi(\eta_2)}+\Theta_{\pi_\supercusp(\eta_2')}-\Ch_{\St_{\SO_4}}$ is stable. Note that
\begin{equation}
\Theta_{\pi(\eta_2)}+\Theta_{\pi_\supercusp(\eta_2)}-\Ch_{\St_{\SO_4}}=c_1D_{(F_{A_1\times\tilde A_1},\mcQ_{A_1\times\tilde A_1}^{F_{A_1\times\tilde A_1}})}+c_2D_{(F_{A_1\times\tilde A_1},\mcQ_1^{F_{A_1\times\tilde A_1}})}\pm q^*\cG_{\sgn}\pm q^*\cG_{\sgn},
\end{equation}
where notations are as in \cite[Definition~5.1.3]{DeBacker-Kazhdan-G2}.
\begin{lemma}\label{unstable-lemma}
The distribution $D_{(F_{A_1\times \tilde A_1},\cG_{\sgn})}$ on $\SO_4(F)$ is not stable. Similarly, no linear combination of the distributions $D_{(F_{A_2},\cG_{\chi'})}$ and $D_{(F_{A_2},\cG_{\chi''})}$ on $\SL_3(F)$ are stable.
\end{lemma}
\begin{proof}
A distribution on $\SO_4(F)$ is stable if and only if it is stable under conjugation by $\PGL_2(F)\times\PGL_2(F)$. Thus all stable distributions on $\SO_4$ must be restricted from invariant distributions on $\PGL_2(F)\times\PGL_2(F)$. But the only invariant distributions on $\PGL_2(F)\times\PGL_2(F)$ are spanned by semisimple orbital integrals, and $D_{(F_{A_1\times \tilde A_1},\cG_{\sgn})}$ is linearly independent from them (as can be seen by evaluating against $\cG_{\sgn}$). An identical argument works for $D_{(F_{A_2},\cG_{\chi'})}$ and $D_{(F_{A_2},\cG_{\chi''})}$.
\end{proof}

Now, since $D_{(F_{A_1\times \tilde A_1},\cG_{\sgn})}$ is not stable, the only linear combination of $\Theta_{\pi(\eta_2)}$ and $\Theta_{\pi_\supercusp(\eta_2)}$ that is stable are those for which $\pm q^*\cG_{\sgn}\pm q^*\cG_{\sgn}=0$ (there are four possibilities). Remark~\ref{omega-eta2-characterization} tells us the only such combinations are $\Theta_{\pi(\eta_2)}+\Theta_{\pi_\supercusp(\eta_2)}-\Ch_{\St_{\SO_4}}$ (one for $\eta_2$ and one for $\eta_2'$). Thus, we have:

\begin{theorem}\label{stability-theorem-2x2}
For ramified quadratic characters $\eta_2$ and $\eta_2'$, the character $\Ch_{\pi(\eta_2)}+\Ch_{\pi_\supercusp(\eta_2')}$ is stable in a neighborhood of $s$ if and only if $\eta_2=\eta_2'$. Thus, $\{\pi(\eta_2),\pi_\supercusp(\eta_2)\}$ is an $L$-packet, for each ramified quadratic character $\eta_2$.
\end{theorem}

\section{Size $3$ mixed packets}

Let $\zeta$ be an order $3$ character of $\F_q^\times$. We will repeatedly use the following Hecke algebra isomorphisms, which is the analogue of Lemma~\ref{quadratic-hecke-iso}. 

\begin{corollary}\label{cubic-hecke-iso}
Let $I$ be the standard Iwahori of $G_2$. 
There exist a canonical support-preserving isomorphism of Hecke algebra
\begin{equation}\label{cubic-hecke-iso1}
\mcH(G_2/\!/I,\zeta^{\pm1}\otimes\zeta^{\pm1})\cong\mcH(\PGL_3/\!/J,\zeta^{\pm1}\circ\det),
\end{equation}
under which the representation $\pi(\eta_3)$ corresponds to the representation $\eta_3^{\pm1}\St_{\PGL_3}$, where $J$ is an Iwahori subgroup of $\PGL_3(F)$. The isomorphism is characterized by the commutative diagram
\begin{equation}
 \begin{tikzcd}
\mcH(T/\!/T_0,\zeta^{\pm1}\otimes\zeta^{\pm1}) \arrow[equal]{r} \arrow[hook]{d}{t_u} & \mcH(T/\!/T_0,\zeta^{\pm1}\circ\det) \arrow[hook]{d}{t_u}\\
\mcH(G_2/\!/I,\zeta^{\pm1}\otimes\zeta^{\pm1})\arrow{r}{\sim}&\mcH(\PGL_3/\!/J,\zeta^{\pm1}\circ\det),
\end{tikzcd}
\end{equation}
where $t_u=t_{\delta_B^{-1/2}}$ is as in \cite[pg~399]{Roche-principal-series}.
\end{corollary}
\begin{proof}
Same proof as in Lemma~\ref{quadratic-hecke-iso}.
\end{proof}
The lemma immediately gives:
\begin{corollary}\label{pro-iwahori-invariants2}
Let $I_+$ be the pro-unipotent radical of the Iwahori subgroup $I$ of $G_2$. Then
\[
\pi(\eta_3)^{I_+}=\zeta\otimes\zeta+\zeta^{-1}\otimes\zeta^{-1}.
\]
\end{corollary}

\subsection{Calculating parahoric invariants for $\pi(\eta_3)$}

\subsubsection{Calculating $\pi(\eta_3)^{G_{\alpha+}}$}

Similar to \S\ref{eta2-parahoric-section}, we have an isomorphism of representations of $G_\alpha/G_{\alpha+}\cong\SL_3(\F_q)$,
\begin{equation}\label{mackey-eta3-alpha}
I_B^{G_2}(\nu\eta_3\otimes\eta_3)^{G_{\alpha+}}\cong\bigoplus_{w\in W/W(\SL_3)}\Ind_{G_\alpha\cap wBw^{-1}/(G_{\alpha+}\cap wBw^{-1})}^{G_\alpha/G_{\alpha+}}(\zeta\otimes\zeta)^w,
\end{equation}

Therefore, the $G_{\alpha+}$-invariants of $I_B^{G_2}(\nu\eta_3\otimes\eta_3)$ gives
\begin{equation}
    I_B^{G_2}(\nu\eta_3\otimes\eta_3)^{G_{\alpha+}}\simeq\Ind_B^{\SL_3}(\zeta^{-1}\otimes1\otimes\zeta)+\Ind_B^{\SL_3}(\zeta^{-1}\otimes1\otimes\zeta).
\end{equation}

Likewise, computing the $G_{\alpha+}$-invariants of $I_{\alpha}$ gives us the following
\begin{align}
    I_{\alpha}(\nu^{1/2}\eta_3\St)^{G_{\alpha+}}&\simeq\Ind_B^{\SL_3}(\zeta^{-1}\otimes1\otimes\zeta)\\
    I_{\alpha}(\nu^{1/2}\eta_3^{-1}\St)^{G_{\alpha+}}&\simeq \Ind_B^{\SL_3}(\zeta^{-1}\otimes1\otimes\zeta).
\end{align}
The representation $\Ind_B^{\SL_3}(\zeta^{-1}\otimes1\otimes\zeta)$ has length $3$ and decomposes into three representations $\chi_{st'}(0)$, $\chi_{st'}(1)$, and $\chi_{st'}(2)$ in the notations of \cite[Table~1b, \S7]{simpson-frame}. These representations are conjugate under conjugation by $\PGL_3(\F_q)$. Similarly, the Deligne-Lusztig induction $R_T^\zeta$, where $T\subset\SL_3(\F_q)$ is an anisotropic torus, decomposes into three cuspidal representations $\chi_{r^2s'}(0)$, $\chi_{r^2s'}(1)$, and $\chi_{r^2s'}(2)$ that form an orbit under conjugation by $\PGL_3(\F_q)$.

The representation $\chi_{st'}(0)$ (resp., $\chi_{r^2s'}(0)$) is characterized by the character value
\[
\Ch_{\chi_{st'}(0)}\begin{pmatrix}1&\theta^\ell\\&1&\theta^\ell\\&&1\end{pmatrix}=\Ch_{\chi_{r^2s'}(0)}\begin{pmatrix}1&\theta^\ell\\&1&\theta^\ell\\&&1\end{pmatrix}=q\delta_{\ell0}-\frac{q-1}3,
\]
where $\theta\in\F_q$ is such that $\theta^3\ne1$.

\begin{definition}\label{omega3-defn}
Let $\eta_3$ be a ramified cubic character of $F^\times$. Then there is a uniformizer $\varpi$ such that $\eta_3(\varpi)=1$. We let
\begin{align}
\omega_\princ^{\eta_3}&:=\chi_{st'}(0)^{\mathrm{diag}(1,1,\varpi)}\\
\omega_\cusp^{\eta_3}&:=\chi_{r^2s'}(0)^{\mathrm{diag}(1,1,\varpi)}
\end{align}
be representations of $G_\alpha/G_{\alpha+}\cong H_\alpha/H_{\alpha+}$.
\end{definition}

\begin{remark}
Note that $\omega_\princ^{\eta_3}=\omega_\princ^{\eta_3^{-1}}$ and $\omega_\cusp^{\eta_3}=\omega_\cusp^{\eta_3^{-1}}$. These are the only overlaps in the definition above.
\end{remark}

\begin{remark}
As in \cite{Digne-Michel-book}, the representations $\omega^{\eta_3}_\princ$ and $\omega^{\eta_3}_\cusp$ are common components of the reducible Deligne-Lusztig induction $R_T^{\zeta}$ and the Gelfand-Graev representation $\Gamma_{\beta,\Oo}$ (notation as in \cite[Thm~4.5]{Barbasch-Moy-LCE}) associated to the nilpotent orbit $\Oo=\Oo_1^1$ (notation as in \cite[\S7.1]{DeBacker-Kazhdan-G2}).

\end{remark}

\begin{proposition}
There is an isomorphism of $G_\alpha/G_{\alpha+}$-representations
\[
\pi(\eta_3)^{G_{\alpha+}}\cong\omega^{\eta_3}_\princ.
\]
\end{proposition}
\begin{proof}
Let $N=I_+/G_{\alpha+}\subseteq G_\alpha/G_{\alpha+}$ be a maximal unipotent subgroup. By Proposition~\ref{pro-iwahori-invariants2}, the $G_\alpha/G_{\alpha+}$-representation $\pi(\eta_2)^{G_{\alpha+}}$ has $N$-invariance $\zeta^{-1}\otimes1\otimes\zeta+\zeta\otimes1\otimes\zeta^{-1}$. Thus
\begin{align}
\pi(\eta_2)^{G_{\beta+}}&= I_\alpha(\nu^{1/2}\eta_3\St)^{G_{\beta+}}\\
&=\Ind_B^{\SL_3}(\zeta^{-1}\otimes1\otimes\zeta)
\end{align}
must be of the form $\chi_{r^2s'}(u)$ for some $u$ (as abstract representations of $\SL_3(\F_q)$), since
\[
\chi_{r^2s'}(u)^N\cong\zeta^{-1}\otimes1\otimes\zeta+\zeta\otimes1\otimes\zeta^{-1}.
\]
Consider the isomorphism Lemma~\ref{quadratic-hecke-iso}
\begin{equation}
        \mathcal{H}(G_2/\!/\mathcal{I},\zeta\otimes 1)\xrightarrow{\sim}\mathcal{H}(\PGL_3/\!/ \mathcal{J},\zeta\circ\det),
\end{equation}
which is support-preserving. Let $\widetilde{\mathcal{J}}:=\mathcal{J}\rtimes \langle\begin{pmatrix}&1\\&&1\\\varpi\end{pmatrix}\rangle$ be the stabilizer of an alcove in the building of $\PGL_3(F)$. Then we have the 
following commutative diagram,
\begin{equation}\label{Roche-square-diagram-PGL3}
    \begin{tikzcd}
        \mathcal{H}(G_2/\!/\mathcal{I},\zeta\otimes\zeta)\arrow[]{r}{\sim}&\mathcal{H}(\PGL_3/\!/ \mathcal{J},\zeta\circ\det)\\
        \mathcal{H}(G_{\alpha}/\!/\mathcal{I},\zeta\otimes\zeta)\arrow[hook]{u}{}\arrow[]{r}{\sim}&\mathcal{H}(\widetilde{\mathcal{J}}/\!/\mathcal{J},\zeta\circ\det)\arrow[hook]{u}{}
    \end{tikzcd}
\end{equation}
The representation $\pi(\eta_3)$ is viewed as a homomorphism $\mcH(G_2/\!/\mathcal I,\zeta\otimes\zeta)\to\C$. Under the top isomorphism we obtain the representation $\eta_3\St_{\PGL_3}$, whose restriction to $\mcH(\widetilde{\mathcal J}/\mathcal J,\zeta\circ\det)$ is the character $\eta_3\circ\det$. Now under the bottom isomorphism we obtain $\omega_\princ^{\eta_3}$, so $\omega_\princ^{\eta_3}$ must be a constituent of $\pi(\eta_3)^{G_{\alpha+}}$.

In fact, by the discussion above, $\pi(\eta_3)^{G_{\alpha+}}\cong\omega_\princ^{\eta_3}$.

\end{proof}

\subsubsection{Calculating $\pi(\eta_3)^{G_{\beta+}}$}

As usual, Mackey theory gives:
\begin{align}
    I_B^{G_2}(\eta_3\otimes\nu\eta_3)^{G_{\beta+}}&=\Ind_B^{\SO_4}(\zeta\otimes\zeta^{-1}\otimes1\otimes1)+\Ind_B^{\SO_4}(\zeta\otimes1\otimes\zeta\otimes1)^2\\
    I_\alpha(\nu^{1/2}\eta_3\St_{\GL_2})^{G_{\beta+}}&=\Ind_{P}^{\SO_4}(\zeta\otimes\zeta^{-1}\otimes\St_{\GL_2})+\Ind_B^{\SO_4}(\zeta\otimes1\otimes\zeta\otimes1)\\
    I_\alpha(\nu^{1/2}\eta_3^{-1}\St_{\GL_2})^{G_{\beta+}}&=\Ind_{P}^{\SO_4}(\zeta^{-1}\otimes\zeta\otimes\St_{\GL_2})+\Ind_B^{\SO_4}(\zeta^{-1}\otimes1\otimes\zeta^{-1}\otimes1).
\end{align}
Thus, as $\SO_4(\F_q)\cong G_\beta/G_{\beta+}$-representations, we have
\[
\pi(\eta_3)^{G_{\beta+}}\subset \Ind_{P}^{\SO_4}(\zeta\otimes\zeta^{-1}\otimes\St_{\GL_2})+\Ind_B^{\SO_4}(\zeta\otimes1\otimes\zeta\otimes1),
\]
where now both summands are irreducible. Moreover, the invariants of these representation with respect to the standard maximal unipotent subgroup $N\subset\SO_4(\F_q)$ gives:
\begin{align}
\Ind_{P}^{\SO_4}(\zeta\otimes\zeta^{-1}\otimes\St_{\GL_2})^N\cong& \zeta\otimes\zeta^{-1}\otimes1\otimes1+\zeta^{-1}\otimes\zeta\otimes1\otimes1\\
\Ind_B^{\SO_4}(\zeta\otimes1\otimes\zeta\otimes1)^N\cong&\zeta\otimes1\otimes\zeta\otimes1+\zeta\otimes1\otimes1\otimes\zeta\\&+1\otimes\zeta\otimes\zeta\otimes1+1\otimes\zeta\otimes1\otimes\zeta.
\end{align}
Thus, by Lemma~\ref{pro-iwahori-invariants2} we must have $\pi(\eta_3)^{G_{\beta+}}\cong \Ind_{P}^{\SO_4}(\zeta\otimes\zeta^{-1}\otimes\St_{\GL_2})$.

\subsubsection{Calculating $\pi(\eta_3)^{G_{\delta+}}$}

Mackey theory gives the isomorphism of $G_\delta/G_{\delta+}\cong G_2(\F_q)$:
\begin{align}
    I_B^{G_2}(\eta_3\otimes\nu\eta_3)^{G_{\delta+}}&=\Ind_{B(\F_q)}^{G_2(\F_q)}(\zeta\otimes\zeta)\\
    I_\alpha(\nu^{1/2}\eta_3^{\pm1}\St_{\GL_2})^{G_{\delta+}}&=\Ind_{P_\alpha(\F_q)}^{G_2(\F_q)}(\zeta^{\pm1}\St_{\GL_2}).
\end{align}
Thus, $\pi(\eta_3)^{G_{\delta+}}$ is the intersection in $\Ind_{B(\F_q)}^{G_2(\F_q)}(\zeta\otimes\zeta)$ of the two sub-representations $\Ind_{P_\alpha(\F_q)}^{G_2(\F_q)}(\zeta\St_{\GL_2})$ and $\Ind_{P_\alpha(\F_q)}^{G_2(\F_q)}(\zeta^{-1}\St_{\GL_2})$, which we denote by $\omega_\princ^\zeta$. In terms of Lusztig's equivalence \cite[Thm~4.23]{Lusztig-characters-Princeton-book}, if $s\in G_2(\F_q)$ is of order $3$ such that $\Cent_{G_2(\F_q)}(s)=\SL_3(\F_q)$, we have
\begin{equation}\label{Lusztig-equiv-PGL3-pi-eta3}
\mathcal E(G_2(\F_q),s)\cong\mathcal E(\PGL_3(\F_q),1),
\end{equation}
and $\omega^\zeta_\princ$ corresponds to $\St_{\PGL_3(\F_q)}$ under \eqref{Lusztig-equiv-PGL3-pi-eta3}. Thus, in conclusion:
\begin{proposition}
Let $\pi(\eta_3)$ be the unique sub-representation of $I(\eta_3\otimes\nu\eta_3)$. Then,
\begin{align}
    \pi(\eta_3)^{G_{\delta+}}&=\omega_\princ^\zeta\\
    \pi(\eta_3)^{G_{\alpha+}}&=\omega_\princ^{\eta_3}\\
    \pi(\eta_3)^{G_{\beta+}}&=\Ind_P^{\SO_4}(\zeta\otimes\zeta^{-1}\otimes\St_{\GL_2})
\end{align}
\end{proposition}

\subsection{The supercuspidal representation $\pi_\supercusp(\eta_3)$}

We consider the following depth-zero supercuspidal representation of $G_2(F)$:
\begin{equation}\label{defn-pi-sc-eta3}
\pi_\supercusp(\eta_3):=\cInd_{G_\alpha}^{G_2}(\omega_\cusp^{\eta_3}).
\end{equation}
By the same argument as in Lemma~\ref{eta2-sc-parahoric}, we obtain
\begin{lemma}
Let $\pi_{\supercusp}(\eta_3)$ be as defined in \eqref{defn-pi-sc-eta3}.
\begin{align}
    \pi_\supercusp(\eta_3)^{G_{\delta+}}&=0\\
    \pi_\supercusp(\eta_3)^{G_{\alpha+}}&=\omega^{\eta_3}_\cusp\\
    \pi_\supercusp(\eta_3)^{G_{\beta+}}&=0.
\end{align}
\end{lemma}

\subsection{Characters on a neighborhood of $1$}\label{3x3-character-section}

Similar arguments as in \S\ref{section: Green-fxn-pi-eta2} gives the following characters for $\pi(\eta_3)$ in terms of Green functions:
\begin{enumerate}
    \item For $F=F_{G_2}$, we have
    \[
    \Ch_{\omega_\princ^{\zeta}}=\frac16(R_{1}^\zeta-3R_{A_1}^\zeta+2R_{A_2}^\zeta),
\]
thus for $u\in G_2(\F_q)$ unipotent, we have $\Ch_{\omega_\princ^{\zeta}}(u)=\frac16(\mcQ_{1}^{F_{G_2}}(u)-3\mcQ_{A_1}^{F_{G_2}}(u)+2\mcQ_{A_2}^{F_{G_2}}(u))$.
\item For $F=F_{A_2}$ we have, for $u\in G_{F}/G_{F+}$ unipotent,
\[
\Ch_{\omega_\princ^{\eta_3}}(u)=\frac13(\mcQ_1^{F_{A_2}}(u)+\omega\cG_{\chi'}(u)+\omega^2\cG_{\chi''}(u))
\]
for some $\omega$ a cube root of unity (uniquely determined by $\eta_3$).
\item For $F=F_{A_1\times\tilde A_1}$, we have
\[
\Ch_{\Ind_P^{\SO_4}(\zeta\otimes\zeta^{-1}\otimes\St_{\GL_2})}=\frac12(R_{1}^\zeta-R_{\tilde A_1}^\zeta),
\]
thus for $u\in G_{F}$ unipotent, we have
\begin{equation}
\Ch_{\Ind_P^{\SO_4}(\zeta\otimes\zeta^{-1}\otimes\St_{\GL_2})}(u)=\frac12(\mcQ_1^{F_{A_1\times\tilde A_1}}(u)-\mcQ_{\tilde A_1}^{F_{A_1\times\tilde A_1}}(u)).
\end{equation} 
\item For $F=F_{A_1}$, we have $\pi(\eta_3)^{G_{F+}}\cong \Ind_B^{\GL_2}(\zeta\otimes\zeta^{-1})$, so on unipotent elements, we have 
\(
\Ch_{\pi(\eta_3)^{G_{F+}}}=\mcQ_{1}^{A_1}.
\)
\item For $F=F_{\tilde A_1}$, we have $\pi(\eta_3)^{G_{F+}}\cong\zeta\St_{\GL_2}+\zeta^{-1}\St_{\GL_2}$, so on unipotent elements, we have $\Ch_{\pi(\eta_3)^{G_{F+}}}=\mcQ_1^{\tilde A_1}-\mcQ_{\tilde A_1}^{\tilde A_1}$.
\item Finally for $F=F_\emptyset$ we have $\pi(\eta_3)^{G_{F+}}=\zeta\otimes\zeta\oplus\zeta^{-1}\otimes\zeta^{-1}$ (as in Corollary~\ref{pro-iwahori-invariants2}), so the character on unipotent elements is $2\mcQ_{\{e\}}^{F_\emptyset}$.
\end{enumerate}
Similarly, for $\pi_\supercusp(\eta_3)$ we have
\begin{equation}
\Ch_{\omega^{\eta_3}_\cusp}(u)=\frac13(\mcQ_{A_2}^{F_{A_2}}(u)+\omega\cG_{\chi'}(u)+\omega^2\cG_{\chi''}(u))
\end{equation}
where $\omega$ is a cube root of unity (uniquely determined by $\eta_3$) and $\cG_{\chi'},\cG_{\chi''}$ are generalized Green functions as in \cite[\S5.2.2]{DeBacker-Kazhdan-G2}. Let $\pi_{\supercusp}(\eta_3)^{\vee}$ denote the dual representation of $\pi_\supercusp(\eta_3)$. 
We have:
\begin{proposition}
All combinations $\pi(\eta_3)+\pi_\supercusp(\eta_3')+\pi_\supercusp(\eta_3'')^\vee$ for any (possibly equal) ramified cubic characters $\eta_3$, $\eta_3'$, and $\eta_3''$ have stable Harish-Chandra characters on the topologically unipotent elements of $G_2$.
\end{proposition}
\begin{proof}
From the discussion above, in the notation of \cite[Table~4]{DeBacker-Kazhdan-G2}, we see that for some explicitly computable\footnote{They are calculable via formulae in \cite{DeBacker-Kazhdan-G2}; for brevity we do not include them here.} constants $c_i$ and some cube roots of unity $\omega_i$ (uniquely determined by $\eta_3$, $\eta_3'$, and $\eta_3''$, respectively),
\begin{align*}
\Ch_{\pi(\eta_3)}&=\frac19c_1(D^\stable_{A_2}+2D^\unst_{A_2})+c_2(\omega_1 D^\stable_{(F_{A_2,\cG_{\chi'}})}+\omega_1^2 D^\stable_{(F_{A_2,\cG_{\chi''}})})-c_3D^\stable_{\tilde A_1}+c_4D^\stable_{\{e\}}\\
\Ch_{\pi_\supercusp(\eta_3')}&=\frac19c_1(D^\stable_{A_2}-D^\unst_{A_2})+c_2(\omega_2 D^\stable_{(F_{A_2,\cG_{\chi'}})}+\omega_2^2D^\stable_{(F_{A_2,\cG_{\chi''}})})\\
\Ch_{\pi_\supercusp(\eta_3'')^\vee}&=\frac19c_1(D^\stable_{A_2}-D^\unst_{A_2})+c_2(\omega_3 D^\stable_{(F_{A_2,\cG_{\chi'}})}+\omega_3^2D^\stable_{(F_{A_2,\cG_{\chi''}})})
\end{align*}
Thus, by \cite[Lemma~6.4.1]{DeBacker-Kazhdan-G2} the sum $\Ch_{\pi(\eta_3)}+\Ch_{\pi_\supercusp(\eta_3')}+\Ch_{\pi_\supercusp(\eta_3'')^\vee}$ is always stable.
\end{proof}

\subsection{Characters on a neighborhood of $s\in G_2$}\label{3x3-character-section-ss} Let $s\in G_2$ be order $3$ such that $\Cent_{G_2}(s)=\SL_3$. The same construction as in \S\ref{character-2x2-semisimple-section} gives rise to invariant distributions $\Theta_{\pi(\eta_3)}$, $\Theta_{\pi_\supercusp(\eta_3)}$, and $\Theta_{\pi_\supercusp(\eta_3)^\vee}$ on the topologically unipotent elements of $\SL_3$ such that they are represented by compatible locally constant functions (for each ramified cubic $\eta_3$). Similar calculations as in \S\ref{character-2x2-semisimple-section} gives:
\begin{theorem}\label{stability-theorem-3x3}
For ramified cubic characters $\eta_3$, $\eta_3'$, and $\eta_3''$, the sum $\Ch_{\pi(\eta_3)}+\Ch_{\pi_\supercusp(\eta_3')}+\Ch_{\pi_\supercusp(\eta_3'')^\vee}$ is stable in a neighborhood of $s$ if and only if $\eta_3=\eta_3'=\eta_3''$. Thus, $\{\pi(\eta_3),\pi_\supercusp(\eta_3),\pi_\supercusp(\eta_3)^{\vee}\}$ is an $L$-packet, for each ramified cubic character $\eta_3$.
\end{theorem}
\begin{proof}
By Lemma~\ref{unstable-lemma} (together with \cite[Lemma~6.4.1]{DeBacker-Kazhdan-G2}), a character on the topologically unipotent locus $(\SL_3(F))_{0+}$ in $\SL_3(F)$ is stable if and only if it is in the span of semisimple orbital integrals. By \cite[Table~1b]{simpson-frame}, for $u\in H_\alpha/H_{\alpha+}$ unipotent, we have
\[
(\omega_\princ^{\eta_3}+\omega_\cusp^{\eta_3}+(\omega_\cusp^{\eta_3})^\vee)(su)=\mcQ_{1}^{F_{A_2}}(u)+2\mcQ_{A_2}^{F_{A_2}}(u),
\]
which is the only linear combination of $\omega_\princ^{\eta_3}$, $\omega_\cusp^{\eta_3}$, and $(\omega_\cusp^{\eta_3})^\vee$ for which the generalized Green functions $\cG_{\chi'}$ and $\cG_{\chi''}$ do not appear. Thus, by \cite[Lemma~5.2.10]{DeBacker-Kazhdan-G2}, the sum $\Ch_{\pi(\eta_3)}+\Ch_{\pi_\supercusp(\eta_3)}+\Ch_{\pi_\supercusp(\eta_3)^\vee}$ is the only stable combination.
\end{proof}

In fact:

\begin{theorem}\label{stability-theorem}
For a ramified cubic character $\eta_3$, the sum $\Ch_{\pi(\eta_3)}+\Ch_{\pi_\supercusp(\eta_3)}+\Ch_{\pi_\supercusp(\eta_3)^\vee}$ is stable. Similarly, for a ramified quadratic character $\eta_2$, the sum $\Ch_{\pi(\eta_2)}+\Ch_{\pi_\supercusp(\eta_2)}$ is stable.
\end{theorem}
\begin{proof}
We have calculated distributions $\Ch_{\pi(\eta_3)}$, $\Ch_{\pi_\supercusp(\eta_3)}$, and $\Ch_{\pi_\supercusp(\eta_3)^\vee}$ (resp., $\Ch_{\pi(\eta_2)}$ and $\Ch_{\pi_\supercusp(\eta_2)}$) on topologically unipotent neighborhoods of $1$ and $s$. A similar (but easier) calculation gives explicit formulae for the distributions on neighborhoods of other (thus arbitrary) topologically semisimple elements $\gamma\in G_2$.

These calculations are enough to prove stability of the characters of $\Ch_{\pi(\eta_2)}+\Ch_{\pi_\supercusp(\eta_2)}$ and $\Ch_{\pi(\eta_3)}+\Ch_{\pi_\supercusp(\eta_3)}+\Ch_{\pi_\supercusp(\eta_3)^\vee}$ on compact elements. By \cite[Theorem~5.2]{casselman} (by an argument similar to \cite[Lemma~9.3.1]{DeBacker-Reeder}), we conclude full stability, i.e.~Property~\ref{property:atomic-stability}.
\end{proof}

\appendix

\section{Character Table of $\SO_4(\F_q)$}
\addtocontents{toc}{\protect\setcounter{tocdepth}{1}}

\subsection{Classifying conjugacy classes in $\SO_4(\F_q)$}
We introduce the following notation:
\begin{itemize}
    \item $c_1(x)=\begin{pmatrix}x\\&x\end{pmatrix}$ where $x\in\F_q^\times$
    \item $c_2(x,\gamma)=\begin{pmatrix}x&\gamma\\&x\end{pmatrix}$ where $x\in\F_q^\times$ and $\gamma\ne0\in\F_q^\times$. When $\gamma=1$ let $c_2(x):=c_2(x,1)$
    \item $c_3(x,y)=\begin{pmatrix}x\\&y\end{pmatrix}$ where $x\ne y\in\F_q^\times$. When $xy=1$ let $c_3(x):=c_3(x,x^{-1})$, where $x\ne\pm1$.
    \item $c_4(z)$ for the matrix with eigenvalues $z$ and $z^q$, for $z\in\F_{q^2}\backslash\F_q$.
\end{itemize}

Moreover, choose and element $\Delta\in\F_q^\times\setminus(\F_q^\times)^2$ and an element $\alpha\in\F_{q^2}^\times$ such that $\alpha^{q-1}=-1$, a choice of which is unique up to scaling by $\F_q^\times$.

\begin{lemma}
Let $q$ be odd. The conjugacy classes in $\SO_4(\F_q)$ are one of:
\begin{enumerate}
    \item $c_1(1)\times c_1(\pm1)$. There are $2$ such conjugacy classes.
    \item $c_1(1)\times c_2(\pm1)$. There are $2$ such conjugacy classes.
    \item $c_1(1)\times c_3(x_2)$ for $x_2\ne\pm1\in\F_{q}^\times$. Since $c_3(x_2)=c_3(x_2^{-1})$ in $\SL_2(\F_q)$, there are $(q-3)/2$ such conjugacy classes.
    \item $c_1(1)\times c_4(z_2)$ for $z_2\in\F_{q^2}\backslash\F_q$ such that $z_2^{q+1}=1$. Since $c_4(z_2)=c_4(z_2^{-1})$ in $\SL_2(\F_q)$ there are $(q-1)/2$ such conjugacy classes.
    \item $c_2(\pm1)\times c_1(1)=c_2(1)\times c_1(\pm1)$. There are $2$ such conjugacy classes.
    \item $c_2(1)\times c_2(\pm1,\gamma_2)$ for $\gamma_2\in\{1,\Delta\}$. There are $4$ such conjugacy classes.
    \item $c_2(1)\times c_3(x_2)$ for $x_2\ne\pm1\in\F_q^\times$. Since $c_3(x_2)=c_3(x_2^{-1})$ in $\SL_2(\F_q)$, there are $(q-3)/2$ such conjugacy classes.
    \item $c_2(1)\times c_4(z_2)$ for $z_2\in\F_{q^2}\backslash\F_q$ with $z_2^{q+1}=1$. Since $c_4(z_2)=c_4(z_2^{-1})$ there are $(q-1)/2$ such conjugacy classes.
    \item $c_3(x_1)\times c_1(1)$ for $x_1\ne\pm1\in\F_q^\times$. Since $c_3(x_1)=c_3(x_1^{-1})$ in $\GL_2(\F_q)$ there are $(q-3)/2$ such conjugacy classes.
    \item $c_3(x_1)\times c_2(1)$ for $x_1\ne\pm1\in\F_q^\times$. Since $c_3(x_1)=c_3(x_1^{-1})$ in $\SL_2(\F_q)$ there are $(q-3)/2$ such conjugacy classes. 
    \item $c_3\times c_3$. 
    There are the following cases:
    \begin{enumerate}
        \item $c_3(x_1)\times c_3(x_2)$ where $x_1^2\ne-1$ or $x_2^2\ne-1$, then since $c_3(x_1)=c_3(x_1^{-1})$ and $c_3(x_2)=c_3(x_2^{-1})$ in $\SL_2(\F_q)$, and $c_3(x_1)\times c_3(x_2)=c_3(-x_1)\times c_3(-x_2)$ there are
    \[\begin{cases}
    \frac{(q-3)^2-4}8&q\equiv1\pmod4\\
    \frac{(q-3)^2}8&q\equiv-1\pmod4
    \end{cases}\]
    such conjugacy classes.
    \item $c_3(x_1,\Delta x_1^{-1})\times c_3(x_2,\Delta x_2^{-1})$ where $x_1,x_2\in\F_q^\times$ and $x_1^2\ne-\Delta$ or $x_2^2\ne-\Delta$. Since $c_3(x_1,\Delta x_1^{-1})=c_3(\Delta x_1^{-1},x_1)$ and $c_3(x_2)=c_3(\Delta x_2^{-1})$ in $\SL_2(\F_q)$ there are
    \[\begin{cases}
    \frac{(q-1)^2}8&q\equiv1\pmod4\\
    \frac{(q-1)^2-4}8&q\equiv-1\pmod4
    \end{cases}\] such conjugacy classes.
    \item $c_3(-1,1)\times c_3(-1,1)$. There is one such conjugacy class.
    \end{enumerate}
    \item $c_3\times c_4$. There are the following cases:
    \begin{itemize}
        \item $c_3(x_1)\times c_4(z_2)$ for $x_1\in\F_q^\times$ and $z\in\F_{q^2}\backslash\F_q$ such that $z_2^{q+1}=1$. 
    \item $c_3(x_1,\Delta x_1^{-1})\times c_4(z_2)$ for $x_1\in\F_q^\times$ and $z_2\in\F_{q^2}$ such that $z_2^{q+1}=\Delta$. Since $c_3(x_1,\Delta x_1^{-1})=c_3(\Delta x_1^{-1},x_1)$ and $c_4(z_2)=c_4(\Delta z_2^{-1})$, there are
    \[
    \begin{cases}
    \frac{q^2-1}4&q\equiv1\pmod4\\
    \frac{(q-1)(q+3)}4&q\equiv-1\pmod4
    \end{cases}
    \]
    such conjugacy classes.
    \end{itemize}

    \item $c_4(z_1)\times c_1(1)$ for $z_1\in\F_{q^2}^1\backslash\{\pm1\}$. There are $(q-1)/2$ such conjugacy classes. 
    \item $c_4(z_1)\times c_2(1)$ for $x,y\in\F_q^\times$ and $z_1\in\F_{q^2}$ with $z_1^{q+1}=1$. There are $(q-1)/2$ such conjugacy classes.
    \item $c_4(z_1)\times c_3(x_2)$ for $x_2\ne\pm1\in\F_q^\times$ and $z_1\in\F_{q^2}\backslash\F_q$ such that $z_1^{q+1}=1$. There are $(q-1)(q-3)/4$ such conjugacy classes.
    \item $c_4(z_1)\times c_3(x_2,\Delta x_2^{-1})$ for $x_2\in\F_q^\times$ and $z_1\in\F_{q^2}^\times$ such that $z_1^{q+1}=\Delta$. There are 
    \[
    \begin{cases}
    \frac{q^2-1}4&q\equiv1\pmod4\\
    \frac{(q-1)(q+3)}4&q\equiv-1\pmod4
    \end{cases}
    \]
    such conjugacy classes.
    \item $c_4(z_1)\times c_4(z_2)$ for $z_1,z_2\in\F_{q^2}\backslash\F_q$ with $(z_1z_2)^{q+1}=1$ and $z_1^{q-1}\ne-1$ or $z_2^{q-1}\ne-1$. The since $c_4(z_1)\times c_4(z_2)=c_4(az_1)\times c_4(az_2)$ for any $a\in\F_q^\times$, and $c_4(z_1)=c_4(z_1^q)$ and $c_4(z_2)=c_4(z_2^q)$ in $\SL_2(\F_q)$.
    \item $c_4(\alpha)\times c_4(\alpha^{-1})$. There is a unique such conjugacy class.
\end{enumerate}
\end{lemma}

\subsection{Classifying representations in $\SO_4(\F_q)$}

Let $\GL_{2,2}(\F_q):=\{(g,h)\in\GL_2(\F_q)\times\GL_2(\F_q):\det(g)=\det(h)\}$. Then there is an isomorphism $\SO_4(\F_q)\cong\GL_{2,2}(\F_q)/\F_q^\times$. Let $\mathbb T$ denote the split maximal torus of $\GL_2(\F_q)$.

Now, the centralizer of a semisimple element $(g,h)\in\GL_{2,2}(\F_q)$ in $\SO_4(\F_q)$ is
\begin{align*}
\Cent_{\SO_4(\F_q)}(g,h)&=\{(s,t)\in\GL_{2,2}(\F_q):(sgs^{-1},tht^{-1})=a(g,h) \text{ for some }a\in\F_q^\times\}/\F_q^\times\\
&=\{(s,t)\in\GL_{2,2}(\F_q):(sgs^{-1},tht^{-1})=\pm(g,h)\}/\F_q^\times,
\end{align*}
where the last equality is by observing $\det(g)=\det(sgs^{-1})=\det(ag)=a^2\det(g)$, so $a=\pm1$. Thus, the centralizer depends on whether $-g$ is conjugate to $g$ and whether $-h$ is conjugate to $h$ under $\GL_2(\F_q)$.

The conjugacy classes of semisimple elements $s=(g,h)$ of $\SO_4(\F_q)$ fall into one of the following possibilities:

\begin{enumerate}
    \item $c_1(1)\times c_1(1)$, then $\Cent_{\SO_4}(s)=\SO_4(\F_q)$. Since unipotent representations are independent of isogenies by \cite[Prop~7.10]{Deligne-Lusztig} we have
        \[
        \mathcal E(\SO_4(\F_q),1)\cong\mathcal E(\PGL_2(\F_q)\times\PGL_2(\F_q),1)=\{1\boxtimes1,1\boxtimes\St_{\PGL_2},\St_{\PGL_2}\boxtimes1,\St_{\PGL_2}\boxtimes\St_{\PGL_2}\}.
        \] 
        
        The representation $1_{\PGL_2}\boxtimes1_{\PGL_2}$ corresponds to the representation $1_{\SO_4}$ and $\St_{\PGL_2}\boxtimes\St_{\PGL_2}$ corresponds to the representation $\St_{\SO_4}$. There are $4$ such representations.
    \item $c_1(1)\times c_1(-1)$, then again $\Cent_{\SO_4}(s)=\SO_4(\F_q)$. The representations in $\mathcal E(\SO_4,s)$ are of the form $\pi\otimes\zeta$ where $\pi\in\mathcal E(\SO_4,1)$ and $\zeta(g,h):=\epsilon(\det(g))$ is the unique order $2$ character of $\SO_4(\F_q)$. There are $4$ such representations.
    \item $c_1(1)\times c_3(x_2)$ for $x_2\ne\pm1\in\F_q^\times$, then $\Cent_{\SO_4}(s)=(\GL_2(\F_q)\times\mathbb T)^1/\F_q^\times\cong\GL_2(\F_q)$. Here, $\GL_2(\F_q)$ has two unipotent representations, $1$ and the Steinberg $\St_{\GL_2(\F_q)}$, of dimensions $1$ and $q$, respectively.
        
        Letting $\mathbb P=(\GL_2\times\mathbb B)^1/\F_q^\times\subset\SO_4(\F_q)$ be the parabolic subgroup with Levi $(\GL_2(\F_q)\times\mathbb T)^1/\F_q^\times$, the representations correspond to $\Ind_{\mathbb P}^{\SO_4}(\chi1_{\GL_2})$ and $\Ind_{\mathbb P}^{\SO_4}(\chi\St_{\GL_2})$, for a character $\chi$ of $\F_q^\times$ with $\chi^2\ne1$.
        
        Note that these are irreducible since the Weyl group action replaces $\chi$ with $\chi^{-1}$. There are a total of $2\cdot(q-3)/2=q-3$ representations.
    \item $c_1(1)\times c_4(z_2)$ then $\Cent_{\SO_4}(s)=(\GL_2(\F_q)\times R_{\F_{q^2}/\F_q}\G_m)^1/\F_q^\times$. This has two cuspidal unipotents, $1_{\PGL_2}$ and $\St_{\PGL_2}$, inflated via $(\GL_2(\F_q)\times R_{\F_{q^2}/\F_q}\G_m)^1/\F_q^\times\to\PGL_2(\F_q)$.
        
        They correspond to representations $1_{\GL_2}\boxtimes\rho_\theta$ of $\GL_2\times\GL_2$, restricted to $\GL_{2,2}$ and factored through $\SO_4$. Here, $\theta$ is a regular character of $\F_{q^2}^\times$ with $\theta|_{\F_q^\times}=1$.
    \item $c_3(x_1,y_1)\times c_3(x_2,y_2)$ for $x_1\ne \pm y_1,x_2\ne\pm y_2\in\F_q^\times$ then $\Cent_{\SO_4}(s)=(\mathbb T\times\mathbb T)^1/\F_q^\times$, the maximal split torus of $\SO_4(\F_q)$. This has a unique unipotent, $1$.
        
        They correspond to induced representations $\Ind_{\mathbb B}^{\SO_4}(\chi_1\otimes\chi_2\otimes\chi_3\otimes\chi_4)$, where $\mathbb B$ is the split Borel subgroup of $\SO_4(\F_q)$, where $\chi_i$ are characters of $\F_q^\times$ with $\chi_1\chi_2\chi_3\chi_4=1$ and $\chi_1^2\ne\chi_2^2$ and $\chi_3^2\ne\chi_4^2$. Here, \[\chi_1\otimes\chi_2\otimes\chi_3\otimes\chi_4(\begin{pmatrix}a'\\&b'\end{pmatrix},\begin{pmatrix}c'\\&d'\end{pmatrix}):=\chi_1(a')\chi_2(b')\chi_3(c')\chi_4(d').\]
        These representations are irreducible since the Weyl group acts by swapping $\chi_1$ with $\chi_2$, and swapping $\chi_3$ with $\chi_4$. 
        The number of such representations is:
        \[
        \begin{cases}
        (q+1)^2+4&q\equiv1\pmod4\\
        (q+1)^2&q\equiv3\pmod4.
        \end{cases}
        \]
    \item $c_3(1,-1)\times c_3(1,-1)$. This has two unipotents, $1$ and $\sgn$.
        
        These are the irreducible components of the length $2$ representation $\Ind_{\mathbb B}^{\SO_4}(1\otimes\epsilon\otimes1\otimes\epsilon)$, where $\epsilon$ is the unique order $2$ character of $\F_q^\times$ and $\chi_1^2\chi_2^2=1$. Explicitly, they are induced representations from the index $2$ subgroup $\SL_2(\F_q)\times\SL_2(\F_q)/\pm1\subset \SO_4(\F_q)$:
        \[
        \omega_{\mathrm{princ}}^+:=\Ind_{(\SL_2\times\SL_2)/\pm1}^{\SO_4}(\omega_e^+\boxtimes\omega_e^+),\omega_{\mathrm{princ}}^-:=\Ind_{(\SL_2\times\SL_2)/\mu_2}^{\SO_4}(\omega_e^+\boxtimes\omega_e^-),
        \]
        in the notation of Remark~\ref{sl2-reducible-principal}. In particular, the restriction to $\SL_2(\F_q)\times\SL_2(\F_q)/\pm1$ is $\omega_e^+\boxtimes\omega_e^+\oplus\omega_e^-\boxtimes\omega_e^-$ and $\omega_e^+\boxtimes\omega_e^-\oplus\omega_e^-\boxtimes\omega_e^+$, respectively.
    \item $c_3(x_1,y_1)\times c_4(z_2)$ where $x_1,y_1\in\F_q^\times$ and $z_2\in\F_{q^2}\backslash\F_q$ with $x_1y_1=z_2^{q+1}$. Then $\Cent_{\SO_4}(s)=(\mathbb T\times R_{\F_{q^2}/\F_q}\G_m)^1/\F_q^\times$. This has a unique unipotent, $1$.
        
        Let $\mathbb P=(\mathbb B\times\GL_2)^1/\F_q^\times\subset\SO_4(\F_q)$ be the parabolic subgroup with Levi $(\mathbb T\times\GL_2(\F_q))^1/\F_q^\times\cong\GL_2(\F_q)$. These are the induced representations $\Ind_{\mathbb B}^{\GL_2}(\chi_1\boxtimes\chi_2)\boxtimes\rho_\theta$ of $\GL_2(\F_q)\times\GL_2(\F_q)$, restricted to $\GL_{2,2}$ and factored through $\SO_4$. Here, $\chi_1$ and $\chi_2$ are characters of $\F_q^\times$ with $\chi_1^2\ne\chi_2^2$ and $\theta$ is a regular character of $\F_{q^2}^\times$, where $\chi_1\chi_2\theta|_{\F_q^\times}=1$. 
    \item $c_4(z_1)\times c_4(z_2)$ where $z_1^{q+1}=z_2^{q+1}$ and $z_1^{q-1}\ne-1$ or $z_2^{q-1}\ne-1$. Here. $\Cent_{\SO_4}(s)=(R_{\F_{q^2}/\F_q}\G_m\times R_{\F_{q^2}/\F_q}\G_m)^1/\F_q^\times$. This has a unique unipotent, $1$.
    
    They correspond to representations $\rho_{\theta_1}\boxtimes\rho_{\theta_2}$ of $\GL_2(\F_q)\times\GL_2(\F_q)$, restricted to $\GL_{2,2}(\F_q)$ and inflated to $\SO_4(\F_q)$. Here, $\theta_1\theta_2|_{\F_q^\times}=1$ and $\theta_1^2$ or $\theta_2^2$ is nontrivial on $\F_{q^2}^1$.
    \item $c_4(\alpha)\times c_4(\alpha^{-1})$. Here $\Cent_{\SO_4}(s)=(R_{\F_{q^2}/\F_q}\G_m\times R_{\F_{q^2}/\F_q}\G_m)^1/\F_q^\times\rtimes\mu_2$. This has two unipotents, $1$ and $\sgn$.
    
    They correspond to the two induced representations 
    \begin{equation}\omega_{\mathrm{cusp}}^+:=\Ind_{\SL_2\times\SL_2/\pm1}^{\SO_4}(\omega_{0}^+\boxtimes\omega_0^+)\quad\text{and}\quad \omega_{\mathrm{cusp}}^-:=\Ind_{\SL_2\times\SL_2/\pm1}^{\SO_4}(\omega_0^+\boxtimes\omega_0^-),
    \end{equation}
    using the notation of Remark~\ref{sl2-reducible-cuspidal}.
\end{enumerate}

\begin{remark}\label{gl2-steinberg-character}
The Steinberg representation of $\GL_2(\F_q)$ has character values:
\begin{center}
\begin{tabular}{ |c|c|c|c| } 
 \hline
$c_1(x)$ & $q$\\ 
 $c_2(x)$ & $0$ \\ 
 $c_3(x,y)$ & $1$ \\ 
 $c_4(z)$&$-1$\\
 \hline
\end{tabular}
\end{center}
\end{remark}

\begin{remark}\label{sl2-reducible-principal}
The principal series representation $\Ind_{\mathbb B}^{\SL_2}(\epsilon\otimes1)$ of $\SL_2(\F_q)$ has length two, and splits as $\omega_e^+\oplus\omega_e^-$, where as usual $\epsilon\ne1$ is the unique order $2$ character of $\F_q^\times$. The character tables are:
\begin{center}
\begin{tabular}{ |c||c| c|} 
 \hline
 &$\omega_e^+$&$\omega_e^-$\\
 \hline
$I_2$ & $\frac{q+1}2$&$\frac{q+1}2$\\ 
 $-I_2$ & $\frac{q+1}2\epsilon(-1)$ &$\frac{q+1}2\epsilon(-1)$\\ 
 $c_2(\pm1,\gamma),\gamma\in\{1,\Delta\}$ & $\frac12(\epsilon(\pm1)+\epsilon(\gamma)\sqrt{\epsilon(-1)q})$&$\frac12(\epsilon(\pm1)-\epsilon(\gamma)\sqrt{\epsilon(-1)q})$ \\ 
 $c_3(x)$&$\epsilon(x)$&$\epsilon(x)$\\
 $c_4(z),z^{q+1}=1$&$0$&$0$\\
 \hline
\end{tabular}
\end{center}
\end{remark}

\begin{remark}\label{sl2-reducible-cuspidal}
Let $\theta_0\ne1$ be the unique order $2$ character of $\F_{q^2}^1$, so the restriction of the cuspidal representation $\rho_{\theta_0}$ of $\GL_2(\F_q)$, restricted to $\SL_2(\F_q)$, splits as $\omega_0^+\oplus\omega_0^-$. The character tables are:
\begin{center}
\begin{tabular}{ |c||c| c|} 
 \hline
 &$\omega_0^+$&$\omega_0^-$\\
 \hline
$I_2$ & $\frac{q-1}2$&$\frac{q-1}2$\\ 
 $-I_2$ & $-\frac{q-1}2\epsilon(-1)$ &$-\frac{q-1}2\epsilon(-1)$\\ 
 $c_2(\pm1,\gamma),\gamma\in\{1,\Delta\}$ & $\pm\frac12(-\epsilon(\pm1)+\epsilon(\gamma)\sqrt{\epsilon(-1)q})$&$\pm\frac12(-\epsilon(\pm1)-\epsilon(\gamma)\sqrt{\epsilon(-1)q})$ \\ 
 $c_3(x)$&$0$&$0$\\
 $c_4(z),z\in\F_{q^2}^1$&$-\theta_0(z)$&$-\theta_0(z)$\\
 \hline
\end{tabular}
\end{center}
\end{remark}

Now, we can calculate the character table for $\SO_4(\F_q)$. Here, we ignore twists of representations by outer automorphisms (coming from $\SO_4\subset\tO_4$), which swaps the two $\GL_2$-factors:

\begin{landscape}\pagestyle{plain}
\adjustbox{scale=0.7,center}{
\begin{tabular}{ |c||c|c|c|c|c|c|c|c| }
 \hline
 \multicolumn{9}{|c|}{Representations of $\SO_4(\F_q)$, cases 1-3} \\
 \hline
 &$1_{\SO_4}$&$\zeta$&$1_{\PGL_2}\boxtimes\St_{\PGL_2}$&$(1_{\PGL_2}\boxtimes\St_{\PGL_2})\otimes\zeta$&$\St_{\SO_4}$&$\St_{\SO_4}\otimes\zeta$&$\Ind_{\mathbb P}^{\SO_4}(\chi1_{\GL_2})$&$\Ind_{\mathbb P}^{\SO_4}(\chi\St_{\GL_2})$\\
 \hline
 $c_1(1)\times c_1(\pm1)$  & $1$&$1$&$q$&$q$&$q^2$&$q^2$&$q+1$&$q(q+1)$\\
$c_1(1)\times c_2(\pm1)$&  $1$ &$1$&$0$&$0$&$0$&$0$&  $1$&$q$\\
 $c_1(1)\times c_3(x_2)$ &$1$&$1$&$1$&$1$&$q$&$q$&$\chi^2(x_2)+\chi^{-2}(x_2)$&$q(\chi^2(x_2)+\chi^{-2}(x_2))$\\
 $c_1(1)\times c_4(z_2)$    &$1$&$1$&$-1$&$-1$&$-q$&$-q$ &$0$&$0$\\
 $c_2(1)\times c_1(\pm1)$&  $1$ &$1$&$q$&$q$&$0$&$0$&  $q+1$ &$0$\\
 $c_2(1)\times c_2(\pm1,\gamma_2)$&$1$ &$1$&$0$&$0$&$0$&$0$& $1$&$0$\\
 $c_2(1)\times c_3(x_2)$& $1$&$1$&$1$&$1$&$0$&$0$& $\chi^2(x_2)+\chi^{-2}(x_2)$&$0$\\
 $c_2(1)\times c_4(z_2)$& $1$ &$1$&$-1$&$-1$&$0$&$0$& $0$ &$0$\\
 $c_3(x_1)\times c_1(1)$& $1$ &$1$&$q$&$q$&$q$&$q$& $q+1$ &$q+1$\\
 $c_3(x_1)\times c_2(1)$& $1$  &$1$&$0$&$0$&$0$&$0$& $1$  &$1$\\
 
 $c_3(x_1,y_1)\times c_3(x_2,y_2)$&$1$  &$\epsilon(x_1y_1)$&$1$&$\epsilon(x_1y_1)$&$1$&$\epsilon(x_1y_1)$& $\chi(x_2y_2^{-1})+\chi(x_2^{-1}y_2)$ &$\chi(x_2y_2^{-1})+\chi(x_2^{-1}y_2)$\\
 $c_3(x_1,y_1)\times c_4(z_2)$& $1$&$\epsilon(x_1y_1)$&$-1$&$-\epsilon(x_1y_1)$&$-1$&$-\epsilon(x_1y_1)$& $0$&$0$\\
 $c_4(z_1)\times c_1(1)$& $1$  &$1$&$q$&$q$&$-q$&$-q$& $q+1$  &$-(q+1)$\\
 $c_4(z_1)\times c_2(1)$& $1$ &$1$&$0$&$0$&$0$&$0$& $1$ &$-1$\\
 $c_4(z_1)\times c_3(x_2,y_2)$& $1$ &$\epsilon(x_2y_2)$&$1$&$\epsilon(x_2y_2)$&$-1$&$-\epsilon(x_1y_1)$& $\chi(x_2y_2^{-1})+\chi(x_2^{-1}y_2)$ &$-\chi(x_2y_2^{-1})-\chi(x_2^{-1}y_2)$\\
 $c_4(z_1)\times c_4(z_2)$& $1$ &$\epsilon(z_1^{q+1})$ &$-1$&$-\epsilon(z_1^{q+1})$&$1$&$\epsilon(z_1^{q+1})$& $0$  &$0$\\
 \hline
 \end{tabular}
 }
 
 Here, the representations $\St_{\PGL_2}\boxtimes1_{\PGL_2}$ and $(\St_{\PGL_2}\boxtimes1_{\PGL_2})\otimes\zeta$ are twists of $1_{\PGL_2}\boxtimes\St_{\PGL_2}$ and $(1_{\PGL_2}\boxtimes\St_{\PGL_2})\otimes\zeta$, respectively, under the unique outer automorphism.

 \adjustbox{scale=0.7,center}{
\begin{tabular}{ |c||c|c|c|c|  }
 \hline
 \multicolumn{5}{|c|}{Representations of $\SO_4(\F_q)$, cases 4-6} \\
 \hline
 &$1_{\GL_2}\boxtimes\rho_\theta$&$\Ind_{\mathbb B}^{\SO_4}(\chi_1\otimes\chi_2\otimes\chi_3\otimes\chi_4)$&$\omega_{\mathrm{princ}}^+$&$\omega_{\mathrm{princ}}^-$\\
 \hline
 $c_1(1)\times c_1(\pm1)$  & $q-1$&$(q+1)^2\chi_1\chi_2(\pm1)$& $\frac{(q+1)^2}2\epsilon(\pm1)$&$\frac{(q+1)^2}2\epsilon(\pm1)$\\
$c_1(1)\times c_2(\pm1)$ &$-1$&$(q+1)\chi_1\chi_2(\pm1)$&  $\frac{q+1}2\epsilon(\pm1)$&$\frac{q+1}2\epsilon(\pm1)$\\
 $c_1(1)\times c_3(x_2)$ &$0$&$(q+1)(\chi_3^{-1}\chi_4(x_2)+\chi_3\chi_4^{-1}(x_2))$&$(q+1)\epsilon(x_2)$&$(q+1)\epsilon(x_2)$\\
 $c_1(1)\times c_4(z_2)$   &$-\theta(z_2)-\theta(z_2^q)$&$0$&$0$&$0$\\
 $c_2(1)\times c_1(\pm1)$&$q-1$&$(q+1)\chi_1\chi_2(\pm1)$&  $\frac{q+1}2\epsilon(\pm1)$&$\frac{q+1}2\epsilon(\pm1)$\\
 $c_2(1)\times c_2(\pm1,\gamma_2)$&$-1$&$\chi_1\chi_2(\pm1)$& $\frac12(\epsilon(\pm1)+\epsilon(-\gamma_2)q)$&$\frac12(\epsilon(\pm1)-\epsilon(-\gamma_2)q)$\\
 $c_2(1)\times c_3(x_2)$&$0$&$\chi_3^{-1}\chi_4(x_2)+\chi_3\chi_4^{-1}(x_2)$&$\epsilon(x_2)$&$\epsilon(x_2)$ \\
 $c_2(1)\times c_4(z_2)$&$-\theta(z_2)-\theta(z_2^q)$&$0$& $0$&$0$\\
 $c_3(x_1)\times c_1(1)$&$q-1$&$(q+1)(\chi_1^{-1}\chi_2(x_1)+\chi_1\chi_2^{-1}(x_1))$& $(q+1)\epsilon(x_1)$&$(q+1)\epsilon(x_1)$\\
 $c_3(x_1)\times c_2(1)$&$1$&$\chi_1^{-1}\chi_2(x_1)+\chi_1\chi_2^{-1}(x_1)$& $\epsilon(x_1)$&$\epsilon(x_1)$\\
$c_3(x_1,y_1)\times c_3(x_2,y_2)$ &$0$&$(\chi_1^{-1}(x_1)\chi_2(y_1)+\chi_1(x_1)\chi_2^{-1}(y_1))(\chi_3^{-1}(x_2)\chi_4(y_2)+\chi_3(x_2)\chi_4^{-1}(y_2))$&$\begin{cases}2\epsilon(x_1x_2)&x_1y_1\in(\F_q^\times)^2\\0&x_1y_1\notin(\F_q^\times)^2\end{cases}$&$\begin{cases}2\epsilon(x_1x_2)&x_1y_1\in(\F_q^\times)^2\\0&x_1y_1\notin(\F_q^\times)^2\end{cases}$\\
 $c_3(x_1,y_1)\times c_4(z_2)$&$-\theta(z_2)-\theta(z_2^q)$&$0$&$0$&$0$\\
 $c_4(z_1)\times c_1(1)$&$q-1$&$0$&   $0$&$0$\\
 $c_4(z_1)\times c_2(1)$&$-1$&$0$& $0$ &$0$\\
 $c_4(z_1)\times c_3(x_2,y_2)$&$0$&$0$& $0$&$0$\\
 $c_4(z_1)\times c_4(z_2)$&$-\theta(z_2)-\theta(z_2^q)$&$0$& $0$&$0$\\
 \hline
 \end{tabular}
 }
 \adjustbox{scale=0.7,center}{
\begin{tabular}{ |c||c|c|c|c| }
 \hline
 \multicolumn{5}{|c|}{Representations of $\SO_4(\F_q)$, cases 7-9} \\
 \hline
 &$\Ind_{\mathbb B}^{\GL_2}(\chi_1\boxtimes\chi_2)\boxtimes\rho_\theta$&$\rho_{\theta_1}\boxtimes\rho_{\theta_2}$&$\omega_{\mathrm{cusp}}^+$&$\omega_{\mathrm{cusp}}^-$\\
 \hline
 $c_1(1)\times c_1(\pm1)$  &$(q^2-1)\theta(\pm1)$&$(q-1)^2\theta_1(\pm1)$&$\pm\frac{(q-1)^2}2\epsilon(\pm1)$&$\pm\frac{(q-1)^2}2\epsilon(\pm1)$\\
$c_1(1)\times c_2(\pm1)$ &$-(q+1)\theta(\pm1)$&$-(q-1)\theta_1(\pm1)$&$\mp\frac{q-1}2\epsilon(\pm1)$&$\mp\frac{q-1}2\epsilon(\pm1)$\\
 $c_1(1)\times c_3(x_2)$ &$0$&$0$&$0$&$0$\\
 $c_1(1)\times c_4(z_2)$    &$-(q+1)(\theta(z_2)+\theta(z_2^q))$&$-(q-1)(\theta_2(z_2)+\theta_2(z_2^q))$&$-(q-1)\theta_0(z_2)$&$-(q-1)\theta_0(z_2)$\\
 $c_2(1)\times c_1(\pm1)$ &$(q-1)\theta(\pm1)$&$-(q-1)\theta_1(\pm1)$&$\mp\frac{q-1}2\epsilon(\pm1)$&$\mp\frac{q-1}2\epsilon(\pm1)$\\
 $c_2(1)\times c_2(\pm1,\gamma_2)$&$-\theta(\pm1)$&$\theta_1(\pm1)$&$\pm\frac12(\epsilon(\pm1)+\epsilon(-\gamma_2)q)$&$\pm\frac12(\epsilon(\pm1)-\epsilon(-\gamma_2)q)$\\
 $c_2(1)\times c_3(x_2)$&$0$&$0$&$0$&$0$\\
 $c_2(1)\times c_4(z_2)$&$-(\theta(z_2)+\theta(z_2^q))$&$\theta_2(z_2)+\theta_2(z_2^q)$&$\frac12\theta_0(z)(1-\sqrt{q^*})$&$\frac12\theta_0(z)(1+\sqrt{q^*})$\\
 $c_3(x_1)\times c_1(1)$ &$(q-1)(\chi_1^{-1}\chi_2(x_1)+\chi_1\chi_2^{-1}(x_1))$&$0$&$0$&$0$\\
 $c_3(x_1)\times c_2(1)$&$\chi_1^{-1}\chi_2(x_1)+\chi_1\chi_2^{-1}(x_1)$ &$0$&$0$&$0$\\
 $c_3(x_1,y_1)\times c_3(x_2,y_2)$ &$0$&$0$&$0$&$0$\\
 $c_3(x_1,y_1)\times c_4(z_2)$&$-(\chi_1(x_1)\chi_2(y_1)+\chi_2(x_1)\chi_1(y_1))(\theta(z_2)+\theta(z_2^q))$ &$0$&$0$&$0$\\
 $c_4(z_1)\times c_1(1)$&$0$&$-(q-1)(\theta_1(z_1)+\theta_1(z_1^q))$&$-(q-1)\theta_0(z_2)$&$-(q-1)\theta_0(z_2)$\\
 $c_4(z_1)\times c_2(1)$&$0$&$\theta_1(z_2)+\theta_1(z_2^q)$&$\frac12\theta_0(z_1)(1-\sqrt{q^*})$&$\frac12\theta_0(z_1)(1+\sqrt{q^*})$\\
 $c_4(z_1)\times c_3(x_2,y_2)$ &$0$&$0$&$0$&$0$\\
 $c_4(z_1)\times c_4(z_2)$
&$0$ &$(\theta_1(z_1)+\theta_1(z_1^q))(\theta_2(z_2)+\theta_2(z_2^q))$&$\begin{cases}0&z_1^{(q+1)/2}\in\F_q^\times\\2\theta_0((z_1z_2)^{(q-1)/2})&z_1^{(q+1)/2}\notin\F_q^\times\end{cases}$&$\begin{cases}0&z_1^{(q+1)/2}\in\F_q^\times\\2\theta_0((z_1z_2)^{(q-1)/2})&z_1^{(q+1)/2}\notin\F_q^\times\end{cases}$\\
 \hline
 \end{tabular}
 }
  Here, we let $q^*:=\epsilon(-1)q\equiv1\pmod4$. The last three representations are cuspidal.
 \end{landscape}

\noindent\textit{Acknowledgements.} 
Y.X.~was supported by NSF grant DMS~2202677. K.S.~was partially supported by MIT-UROP. The authors would like to thank Anne-Marie Aubert, Roman Bezrukavnikov, Stephen DeBacker, Dick Gross, Michael Harris, Tasho Kaletha, Ju-Lee Kim, George Lusztig, Maarten Solleveld, Loren Spice, Minh-T\^{a}m Trinh and Cheng-Chiang Tsai for helpful conversations or correspondences related to this project. The authors would like to thank MIT for providing an intellectually stimulating working environment. 

\bibliographystyle{amsalpha}
\bibliography{bibfile}

\end{document}